\documentclass[hidelinks,onefignum,onetabnum]{siamart220329}

\usepackage{lipsum}
\usepackage{amsmath,amssymb,amsfonts,mathrsfs}
\usepackage{graphicx}
\usepackage[caption=false]{subfig}
\usepackage{epstopdf}
\usepackage{algorithmic}
\usepackage{chemarrow}
\usepackage{textcomp}

\ifpdf
  \DeclareGraphicsExtensions{.eps,.pdf,.png,.jpg}
\else
  \DeclareGraphicsExtensions{.eps}
\fi

\newsiamremark{remark}{Remark}
\newsiamremark{hypothesis}{Hypothesis}
\crefname{hypothesis}{Hypothesis}{Hypotheses}
\newsiamthm{claim}{Claim}
\newsiamthm{example}{Example}

\headers{Chemical oscillator model for module regulation}{X. Shi, C. Gao and D. Dochain}

\title{Chemical relaxation oscillator designed to control molecular computation\thanks{Submmited to the editors DATE.\funding{This work was funded by the National Nature Science Foundation of China under Grant No. 12071428 and 62111530247, and the Zhejiang Provincial Natural Science Foundation of China under Grant No. LZ20A010002.}}}
\author{Xiaopeng Shi\thanks{School of Mathematical Science, Zhejiang University, Hangzhou, PR China
  (\email{12035033@zju.edu.cn}, \email{gaochou@zju.edu.cn}).}
\and Chuanhou Gao\footnotemark[2]
\and Denis Dochain\thanks{ICTEAM, UCLouvain, B ˆatiment Euler, avenue Georges Lemaˆıtre 4-6, 1348 Louvain-la-Neuve, Belgium(\email{denis.dochain@uclouvain.be}).}}

\usepackage{amsopn}

\begin{document}

\maketitle
\begin{abstract}
Embedding efficient calculation instructions into biochemical system has always been a research focus in synthetic biology. One of the key problems is how to sequence the chemical reaction modules that act as units of computation and make them alternate spontaneously. Our work takes the design of chemical clock signals as a solution and presents a $4$-dimensional chemical oscillator model based on relaxation oscillation to generate a pair of symmetric clock signals for two-module regulation. We give detailed dynamical analysis of the model and discuss how to control the period and occurrence order of clock signals. We also demonstrate the loop control of molecular computations and provide termination strategy for them. We can expect that our design for module regulation and loop termination will help advance the embedding of more complicate calculations into biochemical environments.
\end{abstract}



\section{Introduction}
A main objective of synthetic biology is designing programmable chemical controller which can operate in molecular contexts incompatible with traditional electronics \cite{del2016control,vasic2020}. We have learned plenty of algorithms from how real life works such as artificial neural network and genetic algorithm, while on the contrary, inserting advanced computational methods into living organisms to accomplish specific tasks e.g. biochemical sensing and drug delivery remains to be studied. A great deal of related work has sprung up in recent years: Moorman et al. \cite{moorman2019dynamical} proposed a biomolecular perceptron network in order to recognize patterns or classify cells \emph{in-vivo}. The beautiful work of Vasic et al. \cite{vasic2022programming} transformed the feed-forward RRelu neural network into chemical reaction networks (CRNs for short), and performed their model on standard machine learning training sets. There are also some attempts to build CRNs capable of learning \cite{blount2017feedforward,chiang2015reconfigurable}. However, few has implemented the whole neural network computation into a biochemical system. The main reason is that the algorithm based on computer instruction performs operations in a sequential manner whereas biochemical reactions proceed synchronously. This contradiction calls for an appropriate regulation method which isolates two reaction modules \cite{vasic2020} from co-occurring and controls their sequence. Blount et al. \cite{blount2017feedforward} constructed cell-like compartments and added electrical clock signals artificially in order to solve this problem, which increased the difficulty of biochemical implementation. A more natural idea is to design chemical oscillators which can automatically generate chemical species acting as periodical clock signals, taking advantage of the oscillatory changes of their concentration between high and low phases to turn corresponding reaction module on or off.

\par Oscillation phenomena are often encountered in chemical and biological systems e.g. Belousov-Zhabotinskii reaction \cite{tyson1980target} and circadian rhythm \cite{forger2017biological}, research and design of oscillators have been extensively studied in most aspects of industrial applications \cite{castro2004unique,kannapan2016synchronization,zhao2019orbital}, regulating reaction sequence by chemical oscillators is of course of interest. Arredondo and Lakin \cite{arredondo2022supervised} utilized a $20$-dimensional oscillator extended from ant colony model \cite{lachmann1995computationally} to order the parts of their chemical neural networks. Jiang et al. introduced a different oscillator model with $12$ species and $24$ reactions \cite{jiang2011}, then chose two of these species to serve as clock signals. These works follow the same logic: Firstly, find a suitable oscillator model with a set of appropriate selection of parameters and initial values, then confirm that the model is indeed available for use by simulation. There are two main drawbacks with such design. One is the lack of analysis on oscillation mechanism, i.e., it is often unclear why oscillatory behaviour emerges in these models. The other one is the initial concentration of these oscillators seems too strict to be realized in real chemical reaction system, e.g., Arredondo and Lakin demanded the initial concentration of some species equals to $10^{-6}$ while others is $1$ \cite{arredondo2022supervised}. In view of this, we summarize the requirements for species as clock signals and design a universal chemical oscillator model based on relaxation oscillation with clear dynamical analysis i.e. we can explain why the oscillation behaviour occurs and how it evolves, and make sure that the selection of initial values is broad. We also consider the effect of parameter selection on the oscillation properties and provide the period estimation of our clock signals. 

\par In this paper we mainly focus on designing chemical oscillators for the sequence execution of two chemical reaction modules. Sequence control and alternation of two reaction modules are very common in molecular operations and synthetic biology, such as module instructions that involve judgment statement before specific execution, and reaction modules that realize the loop of feed-forward transmission and back-propagation learning process in artificial neural networks. We not only provide a common approach of designing suitable chemical oscillator model for such requirements, but also offer strategy for spontaneous loop termination of reaction modules to be regulated. Our oscillator model can be transformed into abstract chemical reaction networks \cite{feinberg2019foundations} through appropriate kinetics assumption (mainly the \textit{mass-action kinetics}), and finally DNA strand displacement cascades \cite{soloveichik2010dna} or other technical means is used to implement the CRNs into real chemistry.

This paper is organized as follows.  Preliminaries and problem statements are given in \cref{sec:basic}. \Cref{sec:model} exhibits the structure of $4$-dimensional chemical relaxation oscillator based on $2$-dimensional relaxation oscillation, which is able to generate a pair of symmetric clock signals satisfying our requirements for module regulation. We also provide detailed dynamical analysis based on geometric singular perturbation theory on this model. In \cref{sec:fur} we discuss how to control the period and occurrence order of the oscillating species via adjusting oscillator parameters and initial values. Then we demonstrate the loop control of molecular computations and present termination strategy for them in \cref{sec:ter}. And finally, \cref{sec:conclusions} is dedicated to conclusion of the whole paper.

\section{Preliminaries and problem statement}
\label{sec:basic}

In this section we provide the preparatory knowledge on CRN \cite{feinberg2019foundations}, and further formulate the problem by a motivating example of how to construct a chemical oscillator to control the occurrence order of two reaction modules. We first give some notations. The sets of positive integers, real numbers, non-negative real numbers and positive real numbers are denoted by $\mathbb{Z}_{>0}, \mathbb{R}, \mathbb{R}_{\geq 0}$ and $\mathbb{R}_{>0}$, respectively. We use $\mathbb{R}^n$ to denote an $n$-dimensional Euclidean space, a vector $\alpha \in \mathbb{R}^n_{\geq0}$ if any component $\alpha_{i} \in \mathbb{R}_{\geq 0}$ with $i=1,2,...n$ and $\alpha \in \mathbb{R}^n_{>0}$ if $\alpha_{i} \in \mathbb{R}_{> 0}$.
\subsection{CRN}
    A CRN with the $j$th $(j=1,...,m)$ reaction following 
\begin{equation}
R_j:~~~~~a_{j1}S_{1}+\cdots +a_{jn}S_{n} \to  b_{j1}S_{1}+\cdots +b_{jn}S_{n}
\end{equation}    
 consists of three nonempty finite sets $\mathcal{S}$, $\mathcal{C}$ and $\mathcal{R}$, i.e.,
 \begin{itemize}
    \item species set $\mathcal{S}=\{S_1,...,S_n\}$;
    \item complex set $\mathcal{C}=\bigcup_{j=1}^m \{a_{j1}S_{1}+\cdots +a_{jn}S_{n}, b_{j1}S_{1}+\cdots +b_{jn}S_{n}\}$ with each element to be a linear combination of species over the non-negative integers;
    \item reaction set $\mathcal{R}=\{R_1,...,R_m\}$ with each element including two complexes connected by arrow, and the left complex of the arrow called reactant while the right one called product.
\end{itemize}

Denote the concentration of species $S_i$ by $s_i\in\mathbb{R}_{\geq 0}$, then the dynamics describing concentrations change of all species can be written as
\begin{equation}\label{dynamics1}
    \frac{\mathrm{d} s}{\mathrm{d} t} = \Xi \cdot r\left(s\right)\ ,
\end{equation}
where $s\in\mathbb{R}^n_{\geq 0}$ is the $n$-dimensional concentration vector, $\Xi\in\mathbb{Z}^{n\times m}$ is called the stoichiometric matrix with every entry defined by $\Xi_{ij} = b_{ij}-a_{ij}$, and $r(s)$ is the
$m$-dimensional vector-valued function evaluating the reaction rate. The most common model to specify the reaction rate is \textit{mass-action kinetics} that induces $r(s)$ by 
\begin{equation}\label{dynamics2}
r (s) = \left ( k_{1}\prod_{i=1}^{n}s_{i}^{a_{1i}},...,k_{m}\prod_{i=1}^{n}s_{i}^{a_{mi}}  \right )^{\top}
\end{equation}
with $k_{j}>0$ to represent the reaction rate constant of reaction $R_j$. The CRN equipped with \textit{mass-action kinetics} is called mass-action system, which is essentially a group of polynomial ODEs. In the context, we use this class of systems for the subsequent research. The following example gives an illustration of a mass-action system.


\begin{example}
For a reaction network taking the route
\begin{equation*}
    2S_{1} \overset{k_{1}}{\rightarrow} S_{2} + S_{3}\ ,\ \ \ 
    S_{3} \overset{k_{2}}{\rightarrow} 2S_{1}\ ,
\end{equation*}
\end{example}
the species set is $\mathcal{S}=\left \{ S_{1}, S_{2}, S_{3} \right \}$, complex set $\mathcal{C}=\left \{2S_{1}, S_{2}+S_{3},S_{3} \right \}$, stoichiometric matrix $\Xi_{3\times 2} = \begin{pmatrix}
-2&2\\ 
 1&0 \\ 
 1&-1 
\end{pmatrix}$, rate function $r\left (s \right ) = \left ( k_{1}s_{1}^{2}, k_{2}s_{3} \right )^{\top}$, and the corresponding ODEs are: 
\begin{subequations}
\begin{align}
 \frac{\mathrm{d} s_{1}}{\mathrm{d} t} &= -2k_{1}s_{1}^{2} + 2k_{2}s_{3}\ , \\
 \frac{\mathrm{d} s_{2}}{\mathrm{d} t} &= k_{1}s_{1}^{2}\ ,  \\ \frac{\mathrm{d} s_{3}}{\mathrm{d} t} &= k_{1}s_{1}^{2} - k_{2}s_{3}\ .
\end{align}
\end{subequations}

It has been proved that mass-action chemical kinetics is Turing universal \cite{fages2017strong}. This means that any computation can be embedded into a group of polynomial ODEs \cite{bournez2017polynomial}, and then realizing them with mass-action systems. In practice, this process is implemented by mapping the input of calculation into the initial concentrations of some species of the network and the output into the limiting value of other species, usually taking equilibrium. We present an example of ``addition calculation" to give readers more clear illustration \cite{vasic2020}.

\begin{example}\label{ex_gao1}
A CRN follows
\begin{align*}
       S_{1} &\to S_{1} + S_{2}\ , &   
        S_{3} &\to S_{3} + S_{2}\ ,  &           S_{2} &\to \varnothing  
    \end{align*}
 with all the reaction rate constants to be $1$ (In the context, when the reaction rate constant is $1$, we just omit it), where the last reaction refers to an outflow reaction. This network can serve for implementing addition calculation, like $a+b=c,~a,b,c\in\mathbb{R}$. To this task, we write the ODEs of the dynamics as 
\begin{align*}
\frac{\mathrm{d} s_{2}}{\mathrm{d} t} &= s_{1} + s_{3} - s_{2}\ , & 
\frac{\mathrm{d} s_{1}}{\mathrm{d} t} &=\frac{\mathrm{d} s_{3}}{\mathrm{d} t} = 0 
\end{align*}
with initial point vector to be $x(0)$. Clearly, when all reactions reach equilibrium, the equilibrium point vector $s^*$ satisfies $s_1^*=s_1(0)$, $s_2^*=s_1^*+s_3^*$ and $s_3^*=s_3(0)$. Therefore, by letting $s_1(0)=a$, $s_3(0)=b$, and $s_2^*=c$, we realize the addition calculation by this network. 
\end{example}

\subsection{Problem statement}\label{sec2.3}
 When implementing calculation using chemical reactions, the core difficulty is to deal with the contradiction between the sequential execution of calculation instructions and intrinsically parallel occurrence of chemical reactions. This is special true for those compound arithmetics, like loop calculation etc., in which the calculating procedures usually have definite priority. We consider the task to implement the frequently-used loop iteration calculation $s_1=s_1+1$ appearing in many machine learning algorithms through the following CRNs. 

 \begin{example}\label{ex2.2}
Given two reaction modules ($\mathcal{M}$) \begin{align*}
\mathcal{M}_1:
        S_{1} &\to S_{1} + S_{2}\ , &\mathcal{M}_2:  S_{2} &\to S_{1} + S_{2}\ , \\
        S_{3} &\to S_{3} + S_{2}\ ,  &   S_{1} &\to \varnothing\ , \\
        S_{2} &\to \varnothing \ ; 
    \end{align*}
the ODEs are 
\begin{align*}
\mathcal{M}_1:
\frac{\mathrm{d} s_{2}}{\mathrm{d} t} &= s_{1} + s_{3} - s_{2}\ , & 
\mathcal{M}_2: \frac{\mathrm{d} s_{1}}{\mathrm{d} t} &= s_{2} - s_{1}\ , \\
\frac{\mathrm{d} s_{1}}{\mathrm{d} t} &=\frac{\mathrm{d} s_{3}}{\mathrm{d} t} = 0\ ;  &  \frac{\mathrm{d} s_{2}}{\mathrm{d} t} &= 0\ .  
\end{align*}
It is easy to get their solutions to be 
\begin{subequations}\label{solution}
    \begin{align}
        \mathcal{M}_1&: s_{2}(t)=s_{1}(0)+s_{3}(0)-(s_1(0)+s_3(0)-s_2(0))e^{-t}, \\
        \mathcal{M}_2&: s_{1}(t)=s_{2}(0)-(s_2(0)-s_1(0))e^{-t}.
    \end{align}
\end{subequations}
$\mathcal{M}_1$ is actually the network given in Example \ref{ex_gao1}, called addition module, and $\mathcal{M}_2$ finishes the load task, called load module \cite{vasic2020}. When these two modules work independently, $\mathcal{M}_1$ can perform the calculation of $s_2^*=s_1^*+1$ by setting $s_3(0)=1$ while $\mathcal{M}_2$ realizes $s_1^*=s_2(0)=s_2^*$. Moreover, the expressions of solutions (\ref{solution}) imply that both of them converge to equilibrium exponentially. Therefore, alternation and loop of $\mathcal{M}_1$ and $\mathcal{M}_2$ could realize the desired loop iteration calculation $s_1=s_1+1$. However, if we directly put these two reaction modules together, there is strong coupling on the dynamics of $S_{1}$ and $S_{2}$, and their concentrations would increase to infinity in the absence of regulation, which fails to execute the calculation instruction. 
\end{example}

The above example suggests that we need to find a new tool to control strictly and alternatively the ``turn on" and ``turn off" of occurrence of two modules $\mathcal{M}_1$ and $\mathcal{M}_2$. A possible solution is to introduce chemical oscillator that produces periodical signals to control reactions. For this purpose, we modify $\mathcal{M}_1$ and $\mathcal{M}_2$ as follows.

\begin{example}\label{ex2.3}
The modified reaction modules are 
 \begin{align*}
    \tilde{\mathcal{M}}_{1}:
        S_{1} + U &\to S_{1} + S_{2} + U\ , & \tilde{\mathcal{M}}_{2}: S_{2} + V&\to S_{1} + S_{2} + V\ ,\\
        S_{3} + U&\to S_{3} + S_{2} + U\ , & S_{1} + V&\to V\ .\\
        S_{2} + U&\to U\ ; 
    \end{align*}
Here, we introduce two species $U$ and $V$ that are involved in reactions as catalysts. Their participation in reactions, on one hand, will not change themselves as reactions go on, and on the other hand, will not affect the dynamics of the other species appearing in the original modules, i.e., not interfering with the original calculation task of $\mathcal{M}_1$ and $\mathcal{M}_2$. The ODEs of dynamics of the whole network ($\tilde{\mathcal{M}}_{1}$ plus $ \tilde{\mathcal{M}}_{2}$) are expressed as 
\begin{subequations}\label{eq2.4}
    \begin{align}
        \frac{\mathrm{d} s_{1}}{\mathrm{d} t} &= (s_{2} - s_{1})v\ ,  \\
\frac{\mathrm{d} s_{2}}{\mathrm{d} t} &= (s_{1} + s_{3} -s_{2})u\ ,  \\
\frac{\mathrm{d} s_{3}}{\mathrm{d} t} &= 0\ .
    \end{align}
\end{subequations}
From the route, it can be concluded that whether species $U$/$V$ is existing will ``turn on" or ``turn off" $\tilde{\mathcal{M}}_{1}$/$\tilde{\mathcal{M}}_{2}$. Hence, as long as the concentrations of $U$ and $V$ are designed to be a pair of clock signals that oscillate with the same period, they will be able to generate ``loop" so as to control the execution sequence of $\tilde{\mathcal{M}}_{1}$ and $\tilde{\mathcal{M}}_{2}$, and finally to realize the loop iteration calculation $s_1=s_1+1$. \cref{fig1} gives a schematic diagram of a pair of standard clock signals to ``turn on" and ``turn off" alternatively and periodically two reaction modules. 

\begin{figure}[htbp]
  \centering
\includegraphics[width=1\linewidth,scale=1.00]{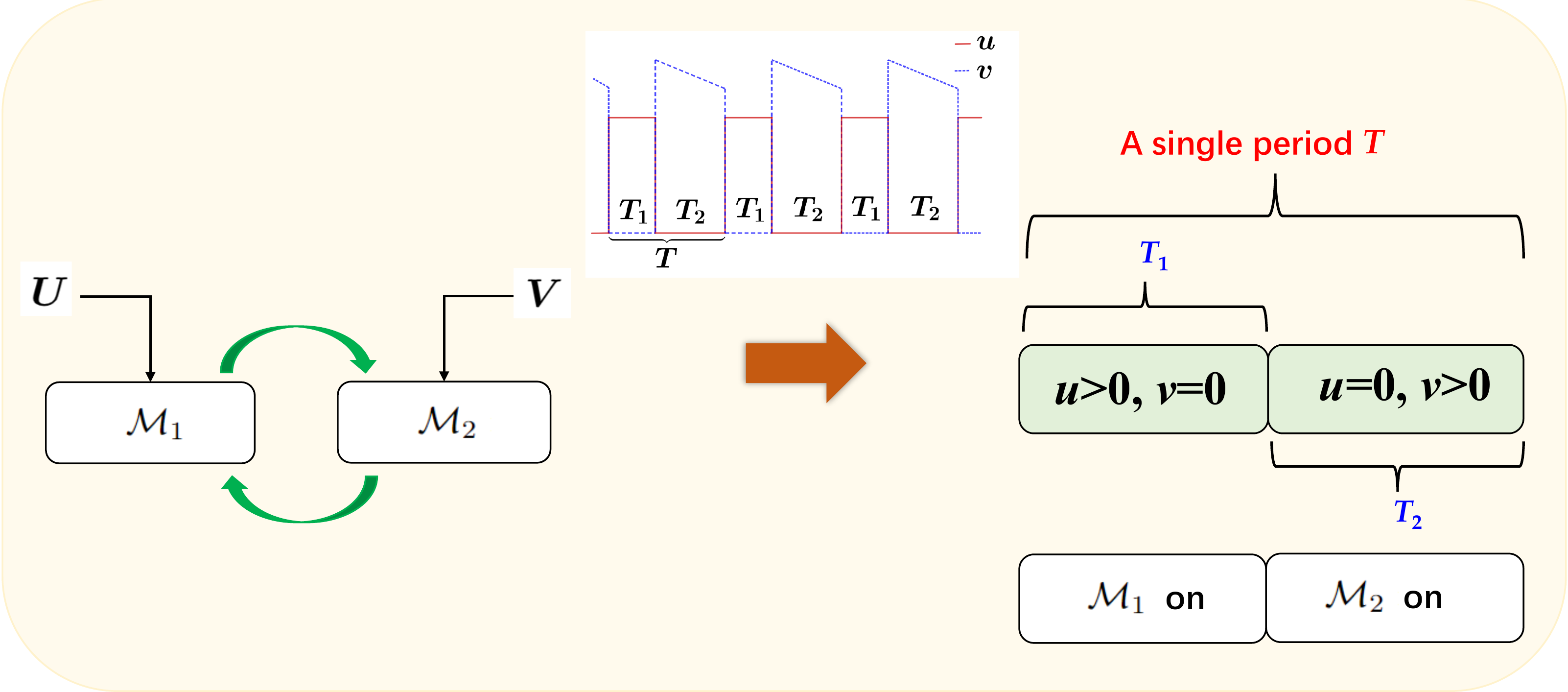}
  \caption{A schematic diagram of a pair of standard clock signals $U$ and $V$ for module regulation.}
  \label{fig1}
\end{figure}
\end{example}


Based on Example \ref{ex2.3} and the usual requirements for clock signals \cite{jiang2011}, we define ours for the current work, called symmetrical clock signals.

\begin{definition}[symmetrical clock signals]\label{def2.4}
    A pair of oscillatory species $U$ and $V$ are called symmetrical clock signals if 
    \begin{enumerate}
        \item $U$ and $V$ oscillate synchronously but with abrupt transitions between their phases;
        \item concentration of $U$ is strictly greater than 0 at high amplitude and approximately 0 at low amplitude, so is $V$;
        \item the amplitudes of $U$ and $V$ are complementary, i.e., when concentration of $U$ is at high amplitude, the one of $V$ is precisely at low amplitude, and vice versa.
    \end{enumerate}
\end{definition}
Note that the last two requirements are trivial and may be generated from any form of oscillation while the first one is not trivial, which serves for guaranteeing the accuracy of module regulation. This motivates us to consider relaxation oscillation \cite{fernandez2020symmetric,field1974oscillations,krupa2013network} as a basic oscillation structure to generate symmetrical clock signals. Thus the following task is on how to develop chemical relaxation oscillator towards controlling molecular computation.


 

\section{Chemical relaxation oscillator}\label{sec:model}
In this section, we introduce the mechanism of relaxation oscillation and develop $4$-dimensional chemical relaxation oscillator for the current task. 

\subsection{Mechanism of 2-dimensional relaxation oscillator}
Relaxation oscillation is a type of common oscillation in biochemical systems \cite{krupa2013network}, whose general form, as an example of $2$-dimensional case, is 
\begin{equation}\begin{aligned}\label{eq2.5}
 \epsilon \frac{\mathrm{d} x}{\mathrm{d} t}  &= f(x,y)\ , \\
\frac{\mathrm{d} y}{\mathrm{d} t}  &=g(x,y)\ ,
   & x \in \mathbb{R},\ y \in \mathbb{R},\ 0 < \epsilon \ll 1\ ,  \end{aligned}\end{equation}
where $f,g$ are $C^k$-functions with $k\geq3$, and $\mathscr{C}_0\triangleq\left \{(x,y):f(x,y)=0\right \}$ is the critical manifold.

\begin{definition}[normally hyperbolic manifold \cite{fenichel1979geometric}]\label{def3.6}
A manifold $\mathscr{C} \subseteq \mathscr{C}_{0}$ is normally hyperbolic if $\forall (x,y) \in \mathscr{C}$, $\frac{\partial f}{\partial x}(x,y) \neq 0$. Point with $\frac{\partial f}{\partial x}(x,y) =0$ is accordingly called non-hyperbolic point or fold point.
   Further, a normally hyperbolic manifold $\mathscr{C}$ is attracting if $\frac{\partial f}{\partial x}(x,y) < 0$ for $\forall (x,y) \in \mathscr{C}$ and is repelling if $\frac{\partial f}{\partial x}(x,y) > 0$.
\end{definition}

There have been a great deal of studies \cite{fernandez2020symmetric,krupa2001relaxation,grasman2012asymptotic,chuang1988asymptotic} on the dynamics of (\ref{eq2.5}), where we are rather concerned with those related to oscillation. When the critical manifold $\mathscr{C}_0$ is S-shaped, the function $y=\varphi(x)$ induced by $\mathscr{C}_0$ has precisely two critical points, a non-degenerate minimum $x_{m}$ and another non-degenerate maximum $x_{M}$ satisfying $x_{M}>x_{m}$, which together with $x_l: \varphi(x_l)=y_M=\varphi(x_M), \ x_r: \varphi(x_r)=y_m=\varphi(x_m)$ 
defines a singular trajectory $\Gamma_{0}$ by
\begin{equation}\begin{aligned}\label{eq_gao3}
 \Gamma_{0}&=\left \{(x,\varphi(x)):x_{l}<x\leq x_{m}\right \} \cup \left \{(x,y_{m}):x_{m}<x\leq x_{r}\right \} \\
 &\cup \left \{(x,\varphi(x)):x_{M}\leq x<x_{r}\right \} \cup \left \{(x,y_{M}):x_{l}\leq x<x_{M}\right \}\ .  \end{aligned}\end{equation}
For this class of systems, Krupa and Szmolyan \cite{krupa2001relaxation} gave a detailed geometric analysis of relaxation oscillation and further suggested a sufficient condition to the existence of a relaxation oscillation orbit $\Gamma_{\epsilon}$ lying in the $O(\epsilon)$-neighborhood of $\Gamma_{0}$. Here, we ignore the details about this condition, and recommend reader to refer to Theorem 2.1 in that paper for more information. Note that in the common model of relaxation oscillation, the oscillation orbit is often accompanied by an unstable equilibrium point on the repelling part of the critical manifold, we continue to use this constraint in our models.

\begin{remark}
The dynamics of (\ref{eq2.5}) serving as a chemical relaxation oscillator should satisfy the flowing three conditions:
\begin{itemize}
    \item [1.] it can generate a relaxation oscillation in the first quadrant;
    \item [2.] its detailed expression should match mass-action kinetics;
    \item [3.] the signals generated should be symmetrical according to Definition \ref{def2.4}.
\end{itemize}
\end{remark}

The van der Pol equation \cite{braaksma1993critical} is a typical instance of structure \cref{eq2.5} with ``S-shaped" manifold, which will act as a basis to design the needed oscillator. Instead of directly using it, we provide a coordinate-transformed version for the current purpose, written as
 \begin{subequations} \label{eq:gao2}
        \begin{align}
            \epsilon \frac{\mathrm{d} x}{\mathrm{d} t} &=-x^3+9x^2-24x+21-y \ ,\\
         \frac{\mathrm{d} y}{\mathrm{d} t} &=x-3 \ , ~~x,~y\in\mathbb{R}_{> 0}
         \ , ~~ 0 < \epsilon \ll 1\ .  
        \end{align}
    \end{subequations} 
Clearly, $(x^*,y^*)=(3,3)$ is its sole equilibrium and the singular trajectory is 
        \begin{equation}\label{Gamma_0}
        \begin{aligned}
\Gamma_{0}=\left\{ \left ( x,\varphi (x) \right ) :1 < x \leq 2 \right\} \cup \left\{ \left ( x,1  \right ):2 < x \leq 5\right\} \\
\cup \left\{ \left ( x,\varphi (x) \right ) : 4 \leq x < 5\right\} \cup \left\{ \left ( x,5 \right ): 1 \leq x < 4\right\}
    \end{aligned}
    \end{equation}
with $\varphi (x)=-x^3+9x^2-24x+21$. The phase plane portrait is presented in \cref{fig2}. As can be seen, the equilibrium is unstable and the relaxation oscillation orbit $\Gamma_{\epsilon}$ which lies in the $O(\epsilon)$-neighborhood of $\Gamma_0$ is also in the first quadrant. Hence, the flows starting from points in the first quadrant, except the unstable equilibrium, will soon converge to $\Gamma_{\epsilon}$ through horizontal motion. Note that the model of (\ref{eq:gao2}) does not match the expressions of the kinetic equations (\ref{dynamics1}) and (\ref{dynamics2}) for a certain mass-action system, since the terms $-y$ and $-3$ cannot reflect the consumption of species $X$ and $Y$, respectively. To fix this point and also avoid destroying the inherent dynamic property of (\ref{eq:gao2}), a naive idea is to multiply the first equation by $x$ while the second equation by $y$, which gives a modified version as follows. 

\begin{figure}[htbp]
  \centering
  \includegraphics[width=1.0\linewidth,scale=1.00]{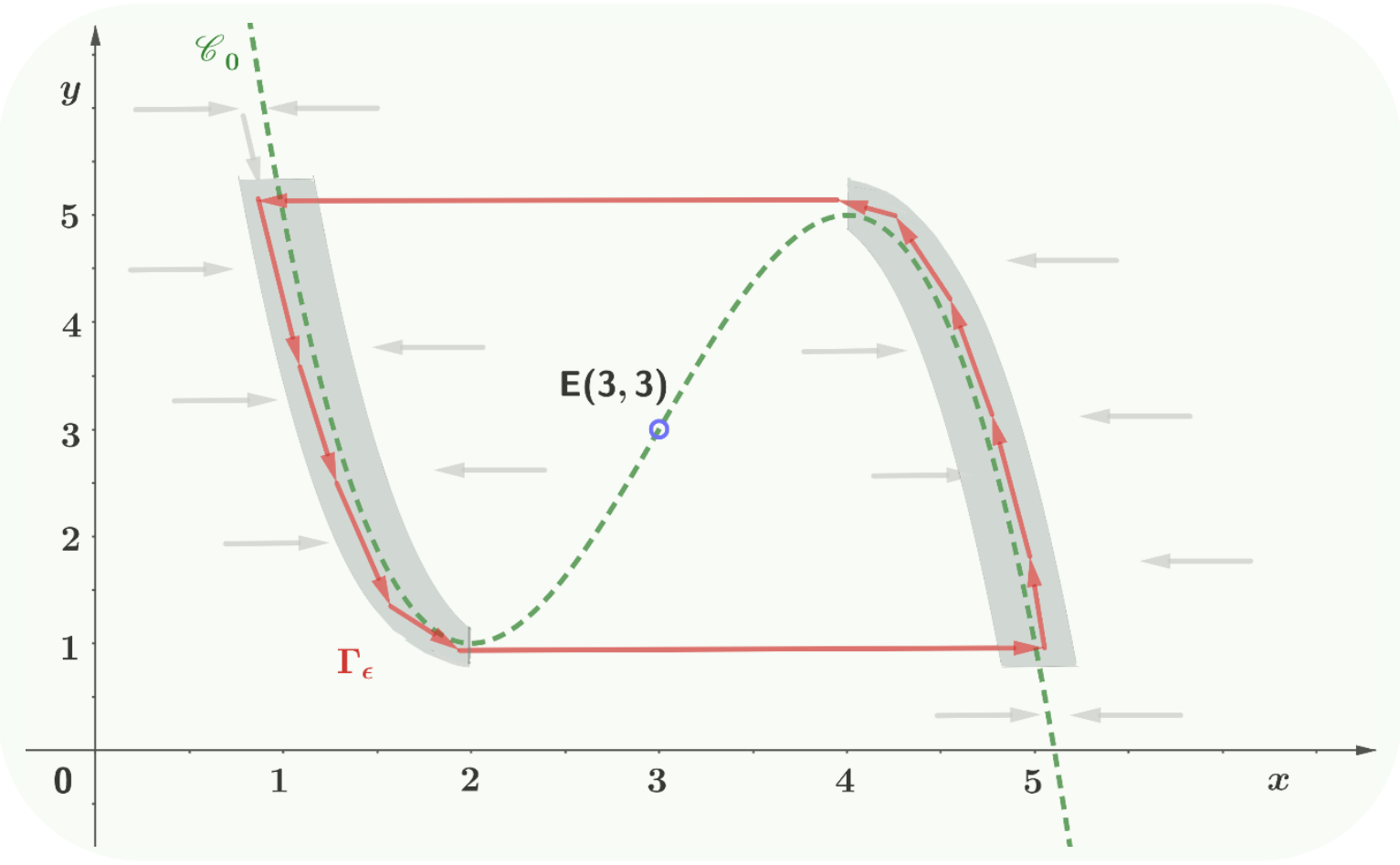}
  \caption{Phase plane portrait of the relaxation oscillation model \cref{eq:gao2}, where the green dotted broken curve represents the critical manifold $\mathscr{C}_0$, the shadow is $\epsilon$-neighbor of the corresponding $\mathscr{C}_0$ part, the rail enclosed by red broken lines approximates position of the relaxation oscillation orbit $\Gamma_{\epsilon}$, the black dotted arrows describe direction of the flows and $E(3,3)$ is the unique unstable equilibrium.}
  \label{fig2}
\end{figure}

\begin{example}[Modified van der Pol model]\label{ex2.6}
The modified van der Pol model is governed by
    \begin{subequations}\label{eq:gao4}
        \begin{align}
            \epsilon \frac{\mathrm{d} x}{\mathrm{d} t} &=(-x^3+9x^2-24x+21-y)x \ ,\\
         \frac{\mathrm{d} y}{\mathrm{d} t} &=(x-3)y \ ,
         ~~x,~y\in\mathbb{R}_{> 0}
         \ , ~~ 0 < \epsilon \ll 1\ .  
        \end{align}
    \end{subequations} 
Compared with (\ref{eq:gao2}), the current model adds $\left \{(x,y):x=0\right \}$ into the critical manifold and creates two saddle points on the axis. \Cref{fig3} displays the oscillating diagram of $x$ and $y$ by taking $(x_{0},y_{0})=(5,5)$ and $\epsilon=0.001$. Obviously, the corresponding species $X$ and $Y$ cannot directly play the role of a pair of symmetrical clock signals as \cref{def2.4} demands. However, the oscillatory species $X$ satisfies the requirement that $x$ has abrupt transitions between phases. We will utilize it and further design other structure to build a pair of symmetrical clock signals. 
\end{example}

\begin{figure}[htbp]
  \centering
  \includegraphics[width=1.0\linewidth,scale=1.00]{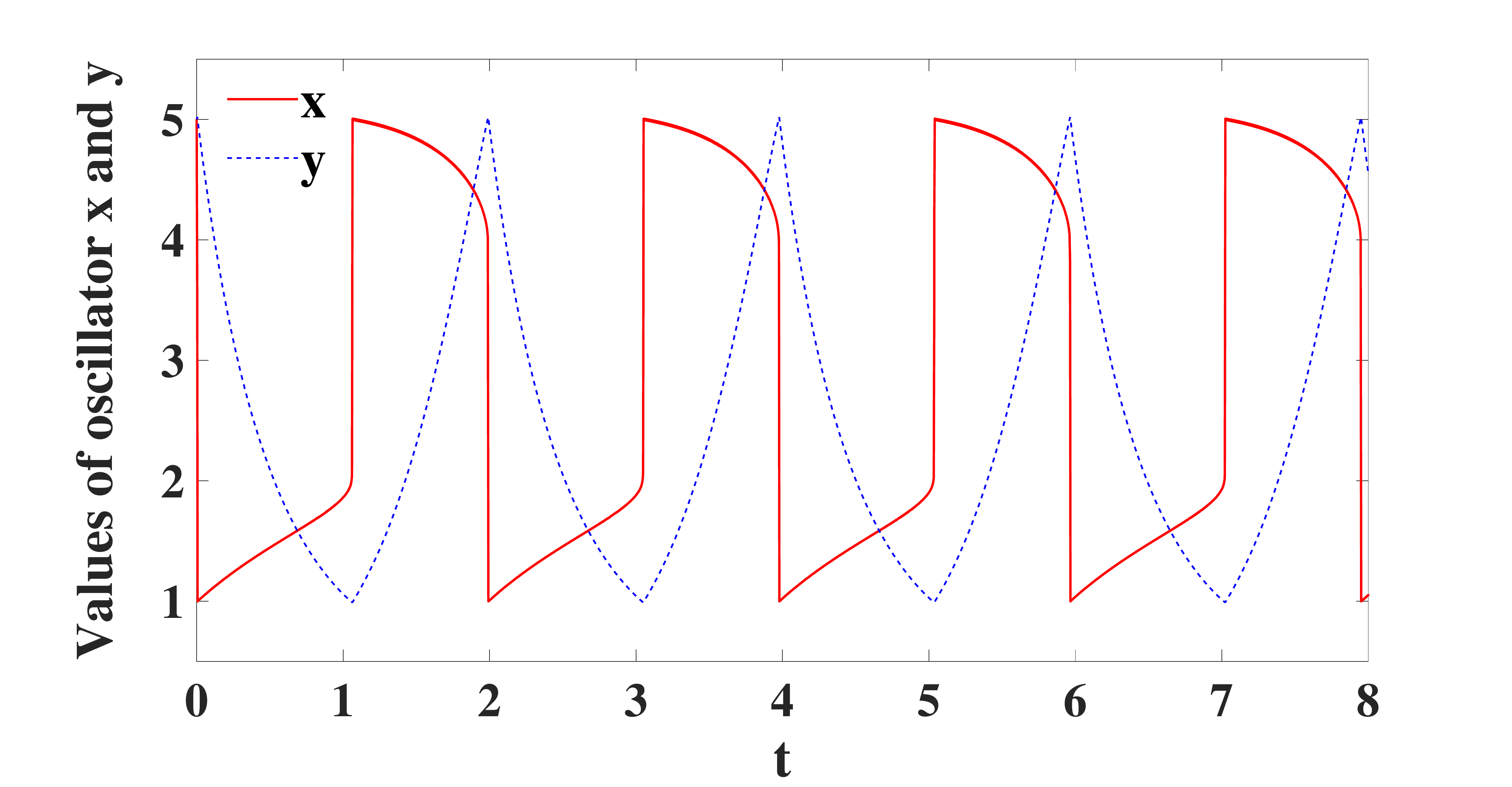}
  \caption{Diagram of oscillators $x$ and $y$ of the model (\ref{eq:gao4}) starting from $(5,5)$ when $\epsilon=0.001$.}
  \label{fig3}
\end{figure}

\subsection{Development of 4-dimensional chemical relaxation oscillator}
As suggested by Example \ref{ex2.6}, the species $X$ from a $2$-dimensional relaxation oscillator can act as a driving signal to produce symmetrical clock signals required in Definition \ref{def2.4}. However, note that this signal is not approximately $0$ at low amplitude, we thus introduce a ``subtraction operation" to pull down its low amplitude to $0$, i.e., considering the known truncated subtraction module \cite{vasic2020,buisman2009computing}
\begin{equation}
\begin{aligned}\label{subtraction}
P &\to P+U\ , & U &\to \varnothing\ , \\
X &\to X+V\ , &  U+V &\to \varnothing\ .
\end{aligned}
\end{equation}
This module may finish the task $u^*=p(0)-x(0)$ when $p(0)>x(0)$ or $u^*=0$ when $p(0)\leq x(0)$. Therefore, as long as the initial concentration of species $P$ is taken be less than that of species $X$, the equilibrium $u^*$ of species $U$ will be ``pulled down" to $0$, i.e., $U$ has the potential to be one of the symmetrical clock signals. Based on this module, we set up the species $X$ to follow the dynamics of (\ref{eq2.5}) exactly, and further modify it to be
\begin{equation}
\begin{aligned}\label{msubtraction}
P &\overset{\kappa}{\rightarrow} P+U\ ,\ \ \ U \overset{\kappa}{\rightarrow} \varnothing\ , \\
X &\overset{\kappa}{\rightarrow} X+V\ , \ \ \ V \overset{\kappa}{\rightarrow} \varnothing\ , \ \ \ U+V \overset{\kappa/\epsilon}{\rightarrow} \varnothing\ 
\end{aligned}
\end{equation}
with $\kappa \gg 1$. Apparently, the modifications include: i) a new outflow reaction $V \overset{\kappa}{\rightarrow} \varnothing$ is added; ii) the reaction rate constant of $U+V \overset{\kappa/\epsilon}{\rightarrow} \varnothing$ is set to be rather large compared to others; iii) the overall reaction rate constants have a significant increase in order of magnitude since $\kappa \gg 1$. We will give reasons of making these modifications during the subsequent analysis. 

By combining the dynamics for the driving signal $X$, i.e., (\ref{eq2.5}), and that of the mass-action system (\ref{msubtraction}), we get
\begin{subequations}\label{eq3.2}
    \begin{align}
        \epsilon_{1}\frac{\mathrm{d} x}{\mathrm{d} t} &= \eta_{1}f(x,y) \label{eq3.2a} \ , \\
         \frac{\mathrm{d} y}{\mathrm{d} t} &= \eta_{1}g(x,y)  \label{eq3.2b} \ ,\\
    \epsilon_{1}\epsilon_{2}\frac{\mathrm{d} u}{\mathrm{d} t} &= \eta_{1}(\epsilon_{1}(p-u)-uv)\label{eq3.2c}\ , \\
    \epsilon_{1}\epsilon_{2}\frac{\mathrm{d} v}{\mathrm{d} t} &= \eta_{1}(\epsilon_{1}(x-v)-uv)\label{eq3.2d}\ 
    \end{align}
\end{subequations}
with $0 < \epsilon_{1}, \epsilon_{2}=\eta_1/\kappa \ll 1$ and $\eta_{1}, p >0$. Note that we reshape the dynamics (\ref{eq2.5}) by multiplying $\eta_{1}$ with the main purpose of distinguishing the time scales of reaction rates between species $X,~Y$ and $U,~V$. The ODEs (\ref{eq3.2a}) and (\ref{eq3.2b}) will degenerate to (\ref{eq2.5}) if $\eta_1$ is modeled into $f(x,y)$ and $g(x,y)$. Here, $f(x,y)$ and $g(x,y)$ are assumed to guarantee the existence of relaxation oscillation \cite{krupa2001relaxation}, and moreover, the relaxation oscillation orbit $\Gamma_{\epsilon}$ lies strictly in the first quadrant along with a unique unstable equilibrium; $p$ is a constant representing the initial concentration of catalyst $P$; and $\eta_1/\epsilon_2\gg \eta_1$ ensures that $U$ and $V$ response quickly enough to the oscillator $X$ so that they can oscillate synchronously with $X$.  

\begin{remark}
 In the ODEs of (\ref{eq3.2c}) and (\ref{eq3.2d}), the very large reaction rate constant for $U+V \to \varnothing$ is set as $\eta_1/\epsilon_1$, which directly borrows the small parameter $\epsilon_1$ for the perturbed system (\ref{eq3.2a}) and (\ref{eq3.2b}). The main reason is only for the convenience of making dynamic analysis, but this is not necessary for developing chemical relaxation oscillator. 
\end{remark}



For this 4-dimensional oscillator model, i.e.,  describing the evolution of species set $\mathcal{S}=\left \{X,Y,U,V\right \}$, there are two time-scale parameters $\epsilon_{1}$ and $\epsilon_{2}$ that motivates us to analyze its dynamics using singular perturbation theory \cite{kuehn2015multiple}. Let $\alpha=(u,v)$, $\beta=(x,y)$, $F(\alpha, \beta)=(\eta_{1}(p-u-uv/\epsilon_{1}),\eta_{1}(x-v-uv/\epsilon_{1}))$ and $G(\alpha, \beta)=(\eta_{1}f(x,y)/\epsilon_{1},\eta_{1}g(x,y))$, then we define the corresponding slow-fast systems (labeled by $\sigma_{sl}$ and $\sigma_{fa}$, respectively) as follows  
\begin{subequations}\label{eq3.3}
    \begin{align}
    \sigma_{sl}&\triangleq\left\{(\alpha,\beta)\bigg| \epsilon_{2}\frac{\mathrm{d} \alpha}{\mathrm{d} t} = F(\alpha, \beta), ~\frac{\mathrm{d} \beta}{\mathrm{d} t} = G(\alpha, \beta)~\text{as}~\epsilon_{2} \to 0 \right\}\ ,\\
  \sigma_{fa}&\triangleq\left\{(\alpha,\beta)\bigg|    \frac{\mathrm{d} \alpha}{\mathrm{d} \tau} = F(\alpha, \beta),~\frac{\mathrm{d} \beta}{\mathrm{d} \tau} = \epsilon_{2}G(\alpha, \beta), ~\tau =t/\epsilon_{2}~\text{as}~\epsilon_{2} \to 0\right\} \ ,
    \end{align}
\end{subequations}
which are equivalent essentially. Further, we define their reduced version by setting $\epsilon_2=0$, i.e.,
\begin{subequations}\label{eq3.5}
    \begin{align}
    \sigma_{rsl}&\triangleq\left\{(\alpha,\beta)\bigg| 0 = F(\alpha, \beta), ~\frac{\mathrm{d} \beta}{\mathrm{d} t} = G(\alpha, \beta) \right\}\label{eq.3.5a}\ ,\\
  \sigma_{rfa}&\triangleq\left\{(\alpha,\beta)\bigg|    \frac{\mathrm{d} \alpha}{\mathrm{d} \tau} = F(\alpha, \beta),~\frac{\mathrm{d} \beta}{\mathrm{d} \tau} = 0, ~\tau =t/\epsilon_{2}\label{eq3.5b}\right\} \ .
    \end{align}
\end{subequations}
The flows generated from $\sigma_{rsl}$ and $ \sigma_{rfa}$ are called slow flow and fast flow, respectively. They will be utilized to approximate the flows of $\sigma_{sl}$ and $ \sigma_{fa}$ under the condition of sufficiently small $\epsilon_2$. The critical manifold $\mathscr{C}_0$, induced by $\sigma_{rsl}$ according to $\mathscr{C}_{0}=\left \{(\alpha, \beta): F(\alpha, \beta)=0 \right \}$, gives
\begin{equation}\label{eq3.7}
      \mathscr{C}_{0}:~~~~~ 
    \begin{aligned}
            \epsilon_{1}(p-u)-uv&=0 \ , \\
       \epsilon_{1}(x-v)-uv&=0 \ ,
    \end{aligned}
\end{equation}


\begin{lemma}\label{le3.7}
The critical manifold $\mathscr{C}_{0}$ given in \cref{eq3.7} is normally hyperbolic and attracting.
\end{lemma}
\begin{proof}
 In this case, $\frac{\partial F}{\partial \alpha}(\alpha, \beta)$ is a $2$-dimensional matrix, and the condition in \cref{def3.6} correspondingly becomes a constraint on eigenvalues \cite{fenichel1979geometric}. From the eigenvalues of $\frac{\partial F}{\partial \alpha}(\alpha, \beta)$, $\lambda_{1}=-\eta_{1}$ and $\lambda_{2}=-\eta_{1}-\eta_{1}(u+v)/\epsilon_{1}$, we have both of them to be less than $0$, so the result is true. 
\end{proof}

By applying the Fenichel Slow Manifold Theorem \cite{fenichel1979geometric} to the above $\mathscr{C}_0$, we have
\begin{remark}\label{slow manifold theorem}
There exists a slow manifold in the $O(\epsilon_2)$-neighborhood of $\mathscr{C}_0$, denoted by $\mathscr{C}_{\epsilon_2}$, satisfying that $\mathscr{C}_{\epsilon_2}$ is also attracting, and moreover, $\mathscr{C}_{\epsilon_2}$ is locally invariant under the flows of $\sigma_{sl}$, i.e., any flow of $\sigma_{sl}$ will remain in motion on the manifold once it enters the neighborhood of $\mathscr{C}_{\epsilon_2}$. $\mathscr{C}_{\epsilon_2}$ can therefore be treated as a perturbation of $\mathscr{C}_0$.
\end{remark}

We can depict the evolution of trajectory of slow-fast system \cref{eq3.3} more concretely through the following theorem, which also implies that the oscillating signals $U$ and $V$ can respond to the driving signal $X$ quickly enough due to the introduction of time scale $\epsilon_2$.  

\begin{theorem}\label{thm3.11}
For the slow-fast system \cref{eq3.3}, any of its trajectories originating from the area $\left \{(\alpha,\beta):\alpha \in \mathbb{R}^2_{\geq 0},\ \beta \in \mathbb{R}^2_{>0} \right \}$ will merge instantaneously into the slow manifold $\mathscr{C}_{\epsilon_2}$ approximately along the fast flow, and moreover, will not leave the manifold.
\end{theorem}
\begin{proof}
   The critical manifold $\mathscr{C}_{0}$ divides the area $\left \{(\alpha,\beta):\alpha \in \mathbb{R}^2_{\geq 0},\ \beta \in \mathbb{R}^2_{>0} \right \}$ into two parts as $F>0$ and $F<0$. Based on \cref{le3.7}, the two eigenvalues of $\frac{\partial F}{\partial \alpha}(\alpha, \beta)$ are always negative at any point in the mentioned area, so fast flows of \cref{eq3.5b} from both sides of $\mathscr{C}_{0}$ tend to travel towards $\mathscr{C}_{0}$, which approximates the instantaneous behavior of the slow-fast system \cref{eq3.3}. Therefore, the trajectory originating from the area will instantaneously converge towards $\mathscr{C}_{0}$ approximately along the fast flows. According to \cref{slow manifold theorem}, $\mathscr{C}_{\epsilon_2}$ lies in the $O(\epsilon_2)$-neighborhood of $\mathscr{C}_{0}$ and is locally invariant, which means the trajectory will finally merge into the slow manifold $\mathscr{C}_{\epsilon_2}$ and will not leave.
\end{proof}

\cref{thm3.11} means that the long-term dynamical behavior of the slow-fast system \cref{eq3.3} is fully decided by $\mathscr{C}_{\epsilon_{2}}$, which essentially results from no non-hyperbolic points on $\mathscr{C}_{0}$. And because $\mathscr{C}_{\epsilon_{2}}$ can be viewed as a perturbation of $\mathscr{C}_{0}$, we just need to pay attention on the dynamical behavior of $\mathscr{C}_{0}$. The following theorem gives an approximation to $\mathscr{C}_{0}$. 


\begin{theorem}\label{thm3.13}
For the critical manifold $\mathscr{C}_{0}$ shown in \cref{eq3.7}, if the initial concentration of catalyst $P$, $p$, is set to be between the high and low amplitude of the driving signal $X$, i.e., $\exists \delta  >0$ such that $x-p>\delta$ when $x$ is at the high amplitude and $p-x>\delta $ when $x$ is at the low amplitude, then the concentrations of oscillating signals $U$ and $V$ can be estimated as
    \begin{equation}\label{eq3.8}
\begin{cases}
u(x)= 0 \ ,\\
v(x)= x-p \ ,
\end{cases}
   \end{equation} 
   when $x$ is at the high amplitude, and
       \begin{equation}\label{eq3.9}
\begin{cases}
u(x)= p-x \ ,\\
v(x)= 0 \ ,
\end{cases}
   \end{equation} 
   when $x$ is at the low amplitude, with each of the estimation errors to be $O(\epsilon_{1})$.
\end{theorem}

\begin{proof}
From \cref{eq3.7}, it is easily to get 
    \begin{subequations}\label{eq2.10}
        \begin{align}
        v^2+(p-x+\epsilon_{1})v-\epsilon_{1}x=0 \ , \\
        u^2-(p-x-\epsilon_{1})u-\epsilon_{1}p=0 \ . 
        \end{align}
    \end{subequations}
We firstly consider the case that $x$ is at the low amplitude, i.e. $\exists \delta  >0$ such that $p-x>\delta$. Let $0<\epsilon_{1} \ll \delta$, then $p-x \pm \epsilon_{1}>0$. Under the constraint of $u,v \geq 0$, the above equations may be solved as
        \begin{subequations}\label{eq3.12}
            \begin{align}
                  u(x)=&\frac{(p-x-\epsilon_{1})+ \sqrt{(p-x-\epsilon_{1})^{2}+4\epsilon_{1}p}}{2} \ , \\
            v(x)=&\frac{-(p-x+\epsilon_{1})+ \sqrt{(p-x+\epsilon_{1})^{2}+4\epsilon_{1}x}}{2} \ .
            \end{align}
        \end{subequations}
Hence, the errors of using (\ref{eq3.9}) to estimate them may be calculated as
        \begin{equation*}
            \begin{aligned}
                \left | u(x)-(p-x) \right |=&\left(\frac{2p}{\sqrt{(p-x-\epsilon_{1})^2+4\epsilon_{1}p}+(p-x-\epsilon_{1})}-1\right)\epsilon_{1} \\
                <&\left(\frac{p}{p-x-\epsilon_{1}}-1\right)\epsilon_{1} 
                <\left(\frac{p}{p-x-\delta}-1\right)\epsilon_{1} 
                =O(\epsilon_{1}) \ 
            \end{aligned}
        \end{equation*}
 and   
               \begin{equation*}
            \begin{aligned}
                \left | v(x)-0 \right |=&\frac{2\epsilon_{1}x}{\sqrt{(p-x+\epsilon_{1})^2+4\epsilon_{1}x}+(p-x+\epsilon_{1})} \\
                <&\frac{x}{p-x+\epsilon_{1}}\epsilon_{1} 
                <\frac{x}{p-x}\epsilon_{1} 
                =O(\epsilon_{1}) \ . 
            \end{aligned}
        \end{equation*}
For the case of $x$ at the high amplitude, the analysis may be performed based on the same logic, which completes the proof.
\end{proof}

\begin{remark}
This theorem indicates that the time-scale parameter $\epsilon_{1}$ appearing in the reaction rate constant of $U+V \overset{\eta_1/\epsilon_1}{\longrightarrow} \varnothing$ plays an important role in generating a pair of oscillating signals with symmetry as required by \cref{def2.4}. It ensures that on the critical manifold $\mathscr{C}_0$ given by (\ref{eq3.7}) there is always a species, either $U$ or $V$, whose concentration is close to $0$ no matter what the concentration of driving signal $X$ is, and furthermore, the approximation error is $O(\epsilon_1)$. This encouraging result also conversely accounts for that the modifications on greatly enlarging the reaction rate constant of $U+V \to \varnothing$ compared to others in (\ref{msubtraction}) is quite reasonable.

\end{remark}

The above two theorems explain how the flows of slow-fast system \cref{eq3.3} evolve towards the neighborhood of $\mathscr{C}_{0}$ and provide a more intuitive description about $\mathscr{C}_{0}$. Now, we can announce that the constructed chemical relaxation oscillator \cref{eq3.2} is qualified to produce symmetrical clock signals as required.

\begin{theorem}\label{thm:3.15}
For a $4$-dimensional system (\ref{eq3.2}) describing the concentrations evolution of species set $\mathcal{S}=\left \{X,Y,U,V\right \}$, assume the initial concentration $p$ of catalyst $P$ to be taken as \cref{thm3.13} claims, and the positive initial concentration point $(x(0),y(0),u(0),v(0))$ of the system to satisfy that $(x(0),y(0))$ is not the unique unstable equilibrium of the subsystem $\Sigma_{xy}$ (governed only by (\ref{eq3.2a}) and (\ref{eq3.2b})). Then species $U$ and $V$ can act as a pair of symmetrical clock signals as requested by \cref{def2.4}.
\end{theorem}
\begin{proof}
  From \Cref{thm3.11}, we know that for the system considered, any trajectory originating from the area $\left \{(x,y,u,v):x,y>0, u,v \geq 0\right \}$ will merge instantaneously into the $O(\epsilon_2)$-neighborhood of critical manifold $\mathscr{C}_{0}$. Further, based on the algebraic equation \cref{eq3.7} defining $\mathscr{C}_0$, concentration of $U$ and $V$ under the conditions will oscillate synchronously with the driving species $X$. Note that abrupt transitions exist between phases of $X$, so do $U$ and $V$. The first item of \cref{def2.4} is satisfied. \Cref{thm3.13} provides an approximate description of $\mathscr{C}_0$, we can utilize \cref{eq3.8,eq3.9} as estimation of concentration of $U$ and $V$ with error $O(\epsilon_{1}+\epsilon_{2})$. Then the second and third items of \cref{def2.4} correspond naturally to the conclusion of \cref{thm3.13}. Therefore, as long as we avoid the unstable equilibrium point as $(x(0),y(0))$, the relaxation oscillator $X$ will drive $U$ and $V$ to oscillate synchronously, acting as the desired symmetrical clock signals.
\end{proof}

At the end of this section, we provide an instance of system \cref{eq3.2} to examine the results of our clock signal design.
\begin{example}[standard chemical relaxation oscillator]\label{ex3.16}
We use the modified van der Pol model presented in \cref{ex2.6} as a specific $2$-dimensional relaxation oscillator to produce the driving signal $X$. Through replacing (\ref{eq3.2a}) and (\ref{eq3.2b}) by (\ref{eq:gao4}) in (\ref{eq3.2}), we obtain the $4$-dimensional chemical oscillator model to be
\begin{subequations}\label{eq3.15}
    \begin{align}
         \epsilon_{1} \frac{\mathrm{d} x}{\mathrm{d} t} &=\eta_{1}(-x^3+9x^2-24x+21-y)x \label{eq3.15a}\ ,\\
         \frac{\mathrm{d} y}{\mathrm{d} t} &=\eta_{1}(x-3)y \label{eq3.15b}\ , \\
         \epsilon_{1}\epsilon_{2}\frac{\mathrm{d} u}{\mathrm{d} t} &= \eta_{1}(\epsilon_{1}(p-u)-uv)\label{eq3.15c}\ , \\
    \epsilon_{1}\epsilon_{2}\frac{\mathrm{d} v}{\mathrm{d} t} &= \eta_{1}(\epsilon_{1}(x-v)-uv)\label{eq3.15d}\ .
    \end{align}
\end{subequations}
Notice that the modified van der Pol model adds $\left \{(x,y):x=0 \right \}$ into the critical manifold induced by the original van der Pol model (\ref{eq:gao2}), but this will not affect the evolution of trajectory from the latter to the former due to the selection of the initial point $(x(0),y(0)) \in \mathbb{R}^2_{>0}$. We name this $4$-dimensional model (\ref{eq3.15}) as the standard chemical relaxation oscillator in the context.
     
With $\epsilon_{1}=\epsilon_{2}=0.001, \eta_{1}=0.1, p=3$ and initial point $(x(0), y(0), u(0), v(0))=(5,5,0,0)$, we show the simulation result of the oscillating species $U$ and $V$ in \cref{fig4}. Obviously, they satisfy the requirements in \cref{def2.4}. 
\end{example}

\begin{figure}[htbp]
  \centering
  \includegraphics[width=1\linewidth,scale=1.00]{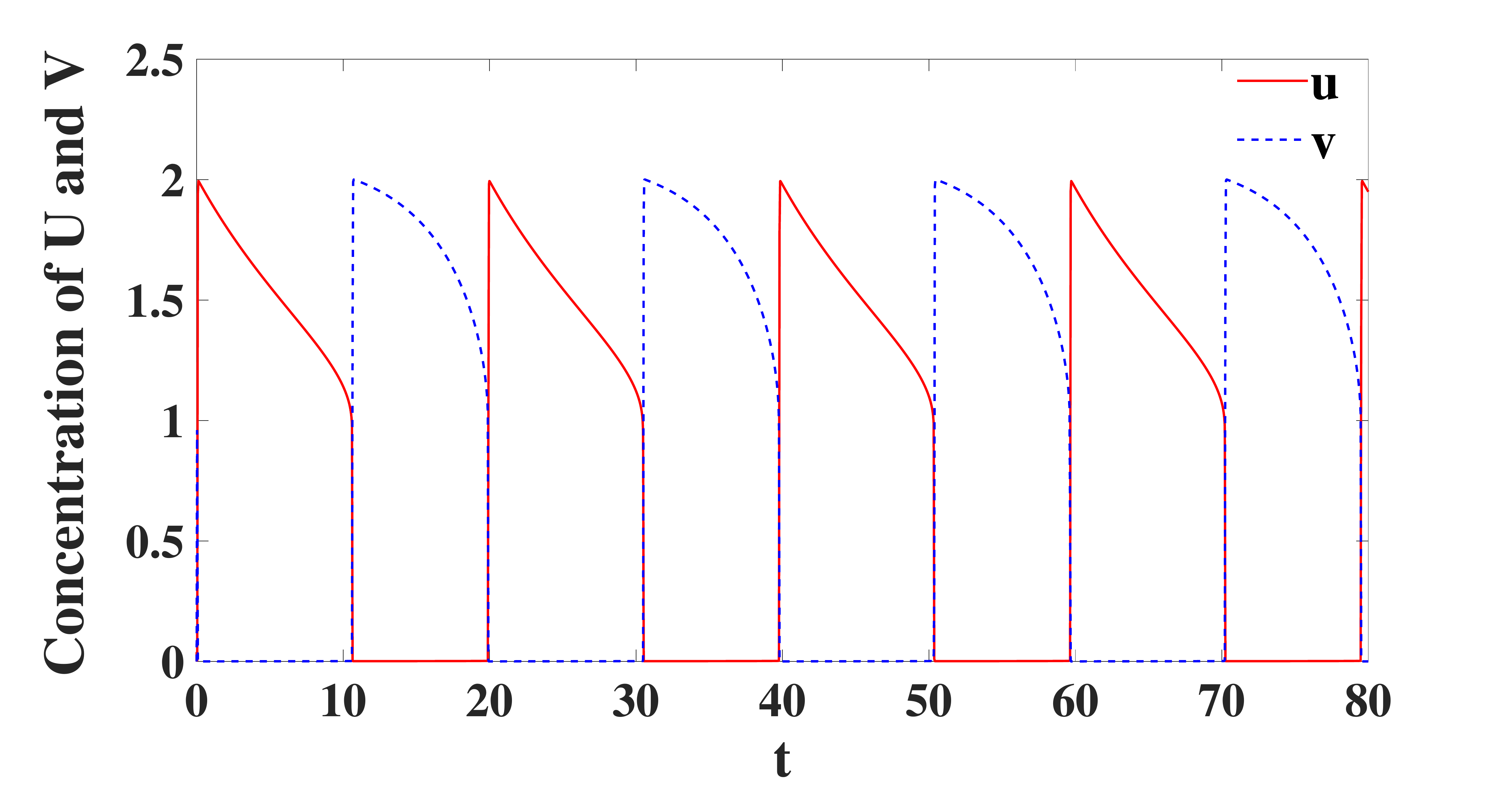}
  \caption{The symmetrical clock signals $U$ and $V$ suggested in \cref{ex3.16}.}
  \label{fig4}
\end{figure}

\section{Period and occurrence order control of symmetrical clock signals}\label{sec:fur}
In this section, based on the standard symmetrical clock signals, we discuss how to control the period and occurrence order of the oscillating species via adjusting oscillator parameters and initial values, respectively.  

\subsection{Oscillator parameters to control the period of symmetrical clock signals}
In the standard chemical relaxation oscillator of \cref{eq3.15}, species $X$ is the driving signal for symmetrical clock signals $U$ and $V$, so we are only concerned with the effect of parameters encountered in \cref{eq3.15a} and \cref{eq3.15b} on them. However, to ensure a S-shaped critical manifold constructed and the unique equilibrium assigned on the repelling part as \cref{fig2} shows, we neglect to discuss those parameters in \cref{eq3.15a}, but keep the current modified van der Pol model (a cubic function) completely. The only parameter in \cref{eq3.15b} that can determine the equilibrium is ``$3$" parameter, we thus redefine the subsystem $\Sigma_{xy}$ of  \cref{eq3.15a} and \cref{eq3.15b} to be 

     \begin{subequations}\label{eq4.1}
         \begin{align}
             \epsilon_{1} \frac{\mathrm{d} x}{\mathrm{d} t} &=\eta_{1}(-x^3+9x^2-24x+21-y)x \ ,\\
    \frac{\mathrm{d} y}{\mathrm{d} t} &=\eta_{1}(x-\ell)y \label{eq4.1b}\ , 
         \end{align}
     \end{subequations}
where $\ell\in \mathbb{R}$ replaces the original ``3", and induces an indefinite equilibrium $x^*=\ell$ for species $X$. Based on the work of Krupa and Szmolyan \cite{krupa2001relaxation}, it can be inferred that in the first quadrant (1) when $\ell<2$ or $\ell>4$, the unique equilibrium $(x^*,y^*)=(\ell, -\ell^3+9\ell^2-24\ell)$ is asymptotically stable, so trajectory will converge to it along the critical manifold and oscillation disappears; (2) when $\ell=2$ or $4$, the equilibrium $(x^*,y^*)=(2,1)$ or $(4,5)$ is actually the non-hyperbolic points on the critical manifold $\mathscr{C}_0$; (3) when $2<\ell<4$, there exists a relaxation oscillation that leads to the instability of the equilibrium. Therefore, the third case $2<\ell<4$ is the available interval, but still in need of avoiding being too close to $2$ or $4$ in practice. Note that in the first quadrant the change of $\ell$ just affects the position of equilibrium of system \cref{eq4.1} along its critical manifold but does not inflect the latter, and the related singular trajectory is exactly the same with $\Gamma_0$ given in \cref{Gamma_0}. 


\begin{lemma}\label{le4.1}
Given a $2$-dimensional relaxation oscillator in the form of \cref{eq4.1}, for its singular trajectory of the first quadrant, i.e., $\Gamma_0$ in \cref{Gamma_0} with $\varphi(x)=-x^3+9x^2-24x+21$, denote the left and right part of $\Gamma_0$ by $\Gamma_{l,0}=\left \{(x,\varphi(x)):1<x<2 \right \}$ and $\Gamma_{r,0}=\left \{(x,\varphi(x)):4<x<5 \right \}$ separately, and their corresponding  relaxation oscillation orbit parts by $\Gamma_{l,\epsilon_{1}}$ and $\Gamma_{r,\epsilon_{1}}$, which can be described as
\begin{subequations}
    \begin{align}
    \Gamma_{l, \epsilon_{1}}: y=& \chi_{l}(x,\epsilon_{1} ) \ , 1<x<2 \ , \\
    \Gamma_{r, \epsilon_{1}}: y=& \chi_{r}(x,\epsilon_{1} ) \ , 4<x<5 \ .
    \end{align}
\end{subequations}
Then for sufficiently small $\epsilon_{0}>0$, there exist differentiable mappings $\psi_{1}$ and $\psi_{2}$ defined separately on $(1,2) \times (0,\epsilon_{0})$ and $(4,5) \times (0,\epsilon_{0})$ as
\begin{subequations}
    \begin{align}
    \psi_{1}&: (x,\epsilon_{1}) \mapsto \varphi(x)-\chi_{l}(x,\epsilon_{1} )\ ,\\
    \psi_{2}&: (x,\epsilon_{1}) \mapsto \chi_{r}(x,\epsilon_{1} )-\varphi(x)\ ,
    \end{align}
\end{subequations}
and moreover, $\left | \psi _{1}(x,\epsilon_{1} ) \right |< O (\epsilon_{1} )\ ,
\left | \psi _{2}(x,\epsilon_{1} ) \right |< O (\epsilon_{1} )\ $.
\end{lemma}
\begin{proof}
$\Gamma_{l,0}$ and $\Gamma_{r,0}$ are two normally hyperbolic parts of the critical manifold, and $\Gamma_{l,\epsilon_{1}}$ and $\Gamma_{r,\epsilon_{1}}$ actually describe the slow manifold corresponding to $\Gamma_{l,0}$ and $\Gamma_{r,0}$. According to Fenichel Slow Manifold Theorem, $\Gamma_{l,\epsilon_{1}}$ and $\Gamma_{r,\epsilon_{1}}$ are separately differential homeomorphic to $\Gamma_{l,0}$ and $\Gamma_{r,0}$, which confirms the existence of differentiable mappings $\psi_{1}$ and $\psi_{2}$. Moreover, $\psi_{1}$ refers to the distance between $\Gamma_{l,\epsilon_{1}}$ and $\Gamma_{l,0}$, which together with the fact that $\Gamma_{l,\epsilon_{1}}$ lies in the $O(\epsilon_1)$-neighborhood of $\Gamma_{l}$, suggests $\left | \psi _{1}(x,\epsilon_{1} ) \right |< O (\epsilon_{1} )$. Based on the same logic, we can get $\left | \psi _2(x,\epsilon_{1} ) \right |< O (\epsilon_{1} )$, too.
\end{proof}

The oscillating period of the driving signal $X$ in \cref{eq4.1} is the time cost to travel along the whole singular trajectory $\Gamma_0$, which can be approximately control by $\ell$ through the following theorem.

\begin{theorem}\label{thm4.2}
For the $2$-dimensional relaxation oscillator \cref{eq4.1}, the oscillating period of $x$ is approximated to be $T=T_{l}+T_{h}$, where $T_{l}$ and $T_{h}$ separately refers to the period of $x$ at low amplitude and high amplitude, given by
    \begin{subequations}\label{eq4.2}
        \begin{align}
            T_{l}&= \int_{1}^{2}\frac{(\varphi'(x)-\frac{\partial \psi _{1}}{\partial x}(x,\epsilon_{1} ))dx}{\eta_{1}(x-\ell)(\varphi(x)-\psi _{1}(x,\epsilon_{1} ))} \label{eq4.2a}\ , \\
            T_{h}&= \int_{5}^{4}\frac{(\varphi'(x)+\frac{\partial \psi _{2}}{\partial x}(x,\epsilon_{1} ))dx}{\eta_{1}(x-\ell)(f(x)+\psi _{2}(x,\epsilon_{1} ))} \label{eq4.2b}\  
        \end{align}
    \end{subequations}
with $\varphi(x)$, $\psi_1(x,\epsilon_1)$ and $\psi_2(x,\epsilon_1)$ to have the same meanings as those in \cref{le4.1}.
\end{theorem}

\begin{proof}
We first confirm the formula for $T_{l}$. According to \cref{le4.1}, we can use $y=\varphi(x)-\psi_{1}(x,\epsilon_{1} )$ to express the orbit when $x$ oscillates at low amplitude, i.e., $\Gamma_{l,\epsilon_{1}}$. Then we get
\begin{equation}\label{eq4.6a}
         \frac{\mathrm{d} y}{\mathrm{d} x}= \varphi'(x)-\frac{\partial \psi _{1}(x,\epsilon_{1} )}{\partial x}\ .
\end{equation}
Furthermore substituting $y=\varphi(x)-\psi_{1}(x,\epsilon_{1} )$ into the right hand of \cref{eq4.1b} yield
\begin{equation}\label{eq4.6b}
\frac{\mathrm{d} y}{\mathrm{d} t}=\eta_{1}(x-\ell)(\varphi(x)-\psi _{1}(x,\epsilon_{1} ))\ .
\end{equation}
Hence the time it takes to travel along $\Gamma_{l,\epsilon_{1}}$ is given by
\begin{equation}\label{eq4.7}
    T_{l}=\int _{\Gamma_{l,\epsilon_{1}}}dt=\int_{1}^{2}\frac{1}{\frac{\mathrm{d} x}{\mathrm{d} t}}dx=\int_{1}^{2}\frac{\frac{\mathrm{d} y}{\mathrm{d} x}}{\frac{\mathrm{d} y}{\mathrm{d} t}}dx 
    =\int_{1}^{2}\frac{(\varphi'(x)-\frac{\partial \psi _{1}}{\partial x}(x,\epsilon_{1} ))dx}{\eta_{1}(x-\ell)(\varphi(x)-\psi _{1}(x,\epsilon_{1} ))}\ .
\end{equation}
In the same way, we can prove $T_{h}$ of \cref{eq4.2b}. Except for $T_l$ and $T_h$, the whole period of $x$ includes the time to travel along the horizontal trajectory parts in the form of fast flow. We thus ignore these extremely short time costs and take $T=T_{l}+T_{h}$ as an approximation of the whole period of $x$.         
\end{proof}

Note that \cref{thm4.2} provides an approximation to the oscillating period of the driving signal $X$, but it is not a practical one. The main reason is the lack of explicit expressions of $\psi _{i}(x,\epsilon_{1}), i=1,2$ in \cref{eq4.2a,eq4.2b}. A more practical estimation is needed. Clearly, in those two formulas $\left | \psi _{i}(x,\epsilon_{1} ) \right | < O(\epsilon_{1}), i=1,2$ are very small for sufficiently small $\epsilon_{1}$, and can thus be considered as negligible; also, $\frac{\partial \psi _{1}(x,\epsilon_{1} )}{\partial x}$ (or $\frac{\partial \psi _{2}(x,\epsilon_{1} )}{\partial x}$) can be thought to be very small to be abandoned since the relaxation oscillation orbit $\Gamma_{l,\epsilon_1}$ (or $\Gamma_{r,\epsilon_1}$) is nearly ``parallel" to $\Gamma_{l,0}$ (or $\Gamma_{r,0}$). We thus give a more simplified but more practical estimation to $T_{l}$ and $T_{h}$ as
\begin{subequations}\label{eq4.8}
    \begin{align}
       T_{l} &\approx  \int_{1}^{2}\frac{\varphi'(x)dx}{\eta_{1}(x-\ell)\varphi(x)} \label{eq4.8a}\ ,  \\
        T_{h} &\approx  \int_{5}^{4}\frac{\varphi'(x)dx}{\eta_{1}(x-\ell)\varphi(x)} \label{eq4.8b}\ .    \end{align}
\end{subequations}

An immediate application of these estimates is to $\Sigma_{xy}$ of the standard chemical relaxation oscillator of \cref{ex3.16}, and by setting $\ell=3$ and $\eta_1=0.1$, we obtain
\begin{subequations}\label{periodEs}
    \begin{align}
         T_{l} &\approx  \int_{1}^{2}\frac{10(-3x^2+18x-24)dx}{(x-3)(-x^3+9x^2-24x+21)}= 10.470\ , \\T_{h} &\approx  \int_{5}^{4}\frac{10(-3x^2+18x-24)dx}{(x-3)(-x^3+9x^2-24x+21)}= 9.193\ .
    \end{align}
\end{subequations}
In the whole standard chemical relaxation oscillator, the oscillation of the driving signal $X$ will stimulate the symmetrical clock signals $U$ and $V$ to oscillate synchronously. Under the parameters of \cref{ex3.16}, when $X$ travels along $\Gamma_{l,\epsilon_1}$, $U$ will travel at high amplitude (Let $T_1$ represent the consumed time); when $X$ travels along $\Gamma_{r,\epsilon_1}$, $V$ will travel at high amplitude ($T_2$ represents the consumed time). Therefore, we have $T_1\approx T_l$ and $T_2\approx T_h$, which are basically consistent with \cref{fig4}. This indicates that the estimations by \cref{eq4.8} are reliable.

\begin{remark}
The estimation formulas \cref{eq4.8} also reveal that, compared to $\ell$, the parameter $\eta_{1}$ can control $T_l$ and $T_h$ more directly. As they actually correspond to the maximum time used for the controlled reaction modular, such as $\tilde{\mathcal{M}}_1$ an $\tilde{\mathcal{M}}_2$ in \cref{ex2.3}, to perform computation, their adjustment can help prolong or shorten the execution time of reaction modular quantitatively. 
\end{remark}

\subsection{Oscillator initial values to control the occurrence of symmetrical clock signals}
Like discussing oscillator parameters, here we only consider the effect of the initial values of species appearing in subsystem $\Sigma_{xy}$ of the standard chemical relaxation oscillator \cref{eq3.15}, i.e., the effect of $(x(0),y(0))$ on oscillating behaviors of $U$ and $V$. As one might know, for subsystem $\Sigma_{xy}$ governed by \cref{eq3.15a,eq3.15b}, the $\omega-limit$ set\footnote{The $\omega-limit$ set refers to the invariant closed set to which the trajectory converges as time $t$ approaches positive infinity. The $\omega-limit$ sets of plane vector fields are usually divided into two categories: closed orbit and equilibrium point.} in the first quadrant only consists of the relaxation oscillation orbit $\Gamma_{\epsilon_{1}}$, i.e., the $O(\epsilon_1)$-neighborhood of $\Gamma_0$ defined in \cref{Gamma_0}, and a unstable equilibrium $(x^*,y^*)=(3,3)$. However, different initial values of $(x(0),y(0))$ will cause the trajectory of $\Sigma_{xy}$ to merge into different positions of $\Gamma_{\epsilon_{1}}$, which finally affects the behavior of $U$ and $V$. 

To depict this effect, we look at the repelling part of its critical manifold in the first quadrant, which takes $\left \{(x,\varphi(x)):~ 2<x<4 \right \}$ with $\varphi(x)=-x^3+9x^2-24x+21$, and is identified by $\varphi_{re}(x)$ for distinguishing from $\varphi(x)$. Through rewriting it as $\left \{(\varphi_{re}^{-1}(y),y):~ 1<y<5 \right \}$, where $\varphi_{re}^{-1}(y)$ means the inverse function of $\varphi(x)$, we can divide the first quadrant of the phase plane of $\Sigma_{xy}$ into two areas 
    \begin{align*}
        \mathcal{A}_1:\left \{(x,y):y\geq5\right \} \cup \left \{(x,y):x<\varphi^{-1}_{re}(y),1<y<5\right \}
        \cup \left \{(\varphi^{-1}_{re}(y),y):3<y<5\right \}\ ,\\
         \mathcal{A}_2:\left \{(x,y):y\leq1\right \} \cup \left \{(x,y):x>\varphi^{-1}_{re}(y),1<y<5\right \}  \cup \left \{(\varphi^{-1}_{re}(y),y):1<y<3\right \}\ ,     
    \end{align*}
where only the unstable equilibrium $(x^*,y^*)=(3,3)$ is excluded. 

\begin{proposition}\label{pro4.4}
Given the subsystem $\Sigma_{xy}$ governed by \cref{eq3.15a,eq3.15b}, any of its trajectories starting from an initial point $(x(0),y(0))\in \mathcal{A}_1$ merges into the left part of relaxation oscillation orbit $\Gamma_{\epsilon_{1}}$, the $O(\epsilon_1)$-neighborhood of $\Gamma_0$ given in \cref{Gamma_0}. The situation changes to the right part of $\Gamma_{\epsilon_{1}}$ if $(x(0),y(0))\in \mathcal{A}_2$.
\end{proposition}
\begin{proof}
    We focus on providing a statement about $\mathcal{A}_1$. The first part of $\mathcal{A}_1$ corresponds to area above $\Gamma_{\epsilon_{1}}$, originating from which the trajectory can only converge horizontally towards the neighborhood of left part of the critical manifold and then merge into the left part of $\Gamma_{\epsilon_{1}}$, i.e., $\Gamma_{l,\epsilon_1}$, along the slow manifold. The trajectory originating from the second part will merge instantaneously into $\Gamma_{l,\epsilon_1}$ because of the effect of the repelling manifold $\varphi_{re}(x)$. The third part describes the upper segment of $\varphi_{re}(x)$, from which the trajectory will immediately enter the second part of $\mathcal{A}_1$ for that $\frac{\mathrm{d} y}{\mathrm{d} t}>0$. The situation for $(x(0),y(0))\in \mathcal{A}_2$ is just inverse. 
\end{proof}

\begin{remark}
\cref{pro4.4} suggests that as long as the initial point of $\Sigma_{xy}$ governed by \cref{eq3.15a,eq3.15b} gets around the unique unstable equilibrium $(3,3)$, the oscillation takes place. Furthermore, to change its position in $\mathcal{A}_1$ or $\mathcal{A}_2$ may lead to the initial oscillation of $x$ at high amplitude or at low amplitude. Frankly, this is not a strict request on the initial point, which is quite different from the one proposed by Arredondo and Lakin \cite{arredondo2022supervised}. They required initial concentration of some species to differ by an order of $10^{6}$ from others. This is not easy to do in real chemical reaction realization. Therefore the current oscillator model is quite encouraging and competitive.
\end{remark}

\begin{figure}[htbp]
  \centering
\includegraphics[width=1\linewidth,scale=1.00]{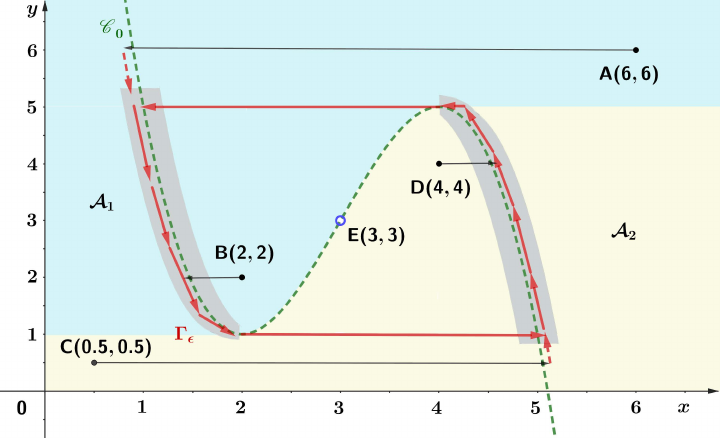}
  \caption{Trajectory evolution diagram of subsystem $\Sigma_{xy}$ of \cref{eq3.15} starting from four different initial points $A(6,6),~B(2,2)\in\mathcal{A}_1$ and $C(0.5,0.5,~D(4,4)\in\mathcal{A}_2$, where the green, red and black lines express the same information as given in \cref{fig2}.}
  \label{fig5}
\end{figure}

\begin{figure}[tbhp]
\centering
\subfloat[$x(0)=y(0)=6$.]{\label{fig6a}\includegraphics[width=0.5\linewidth]{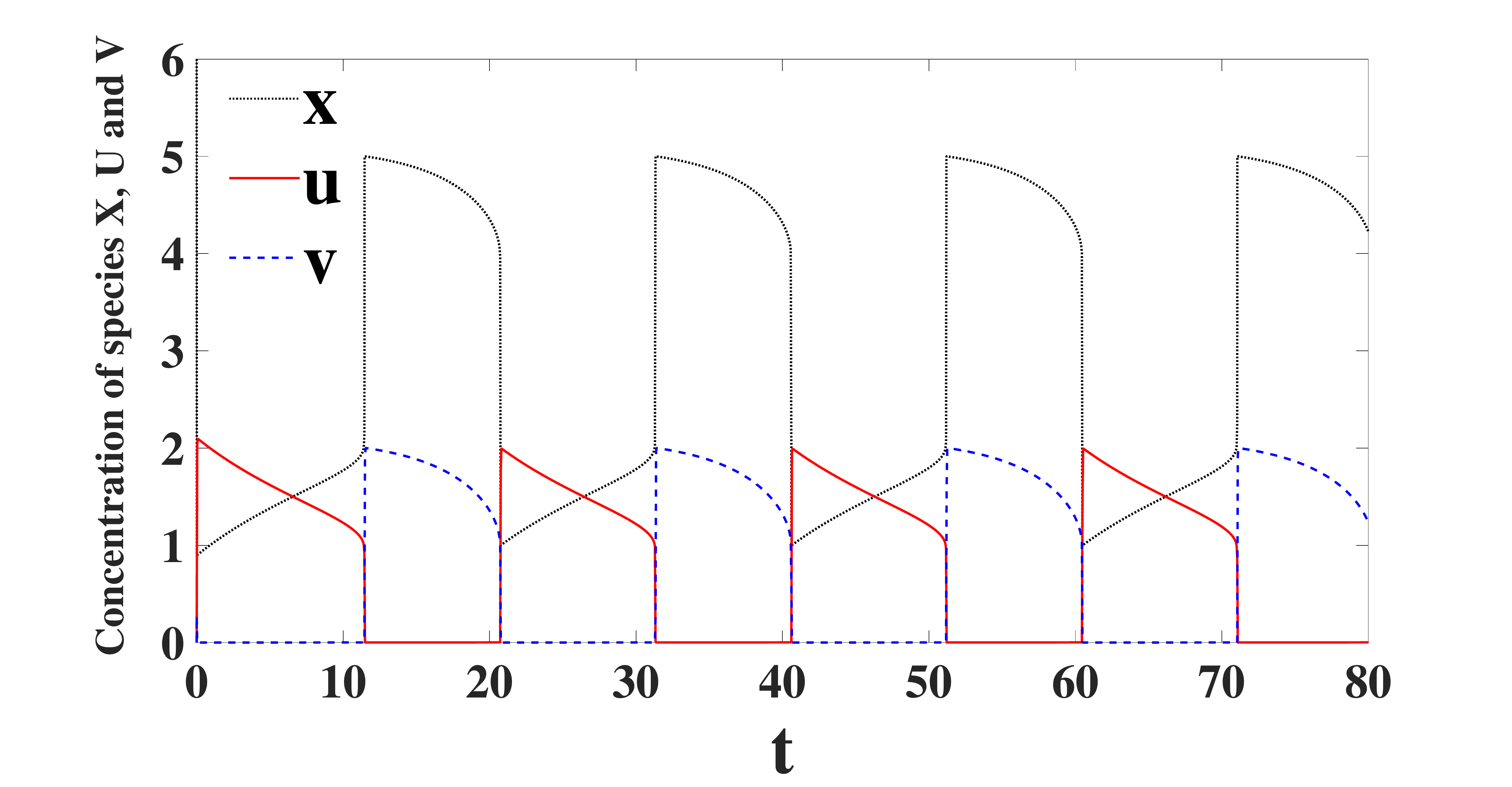}}
\subfloat[$x(0)=y(0)=2$.]{\label{fig6b}\includegraphics[width=0.5\linewidth]{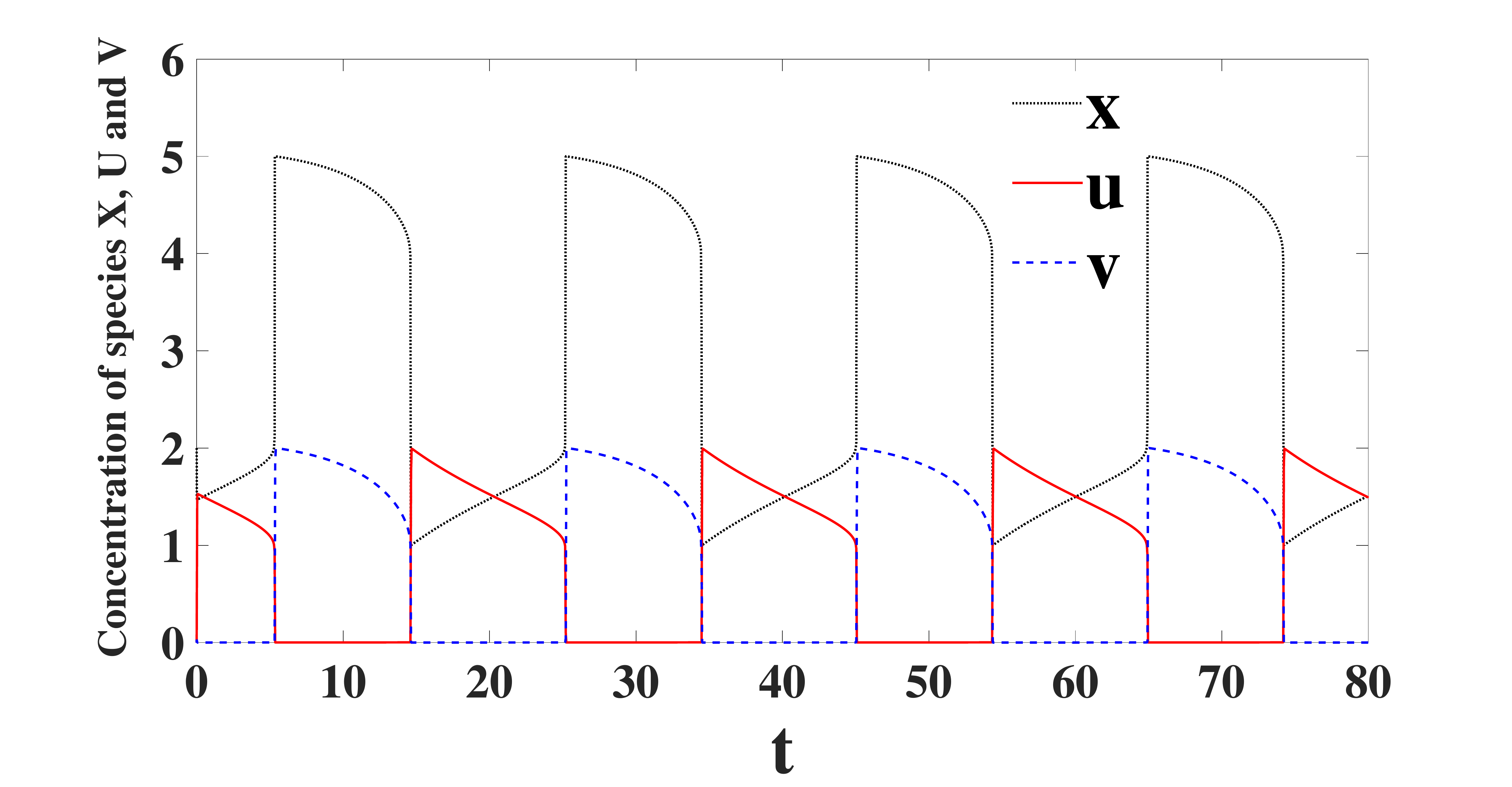}}
\qquad
\subfloat[$x(0)=y(0)=0.5$.]{\label{fig6c}\includegraphics[width=0.5\linewidth]{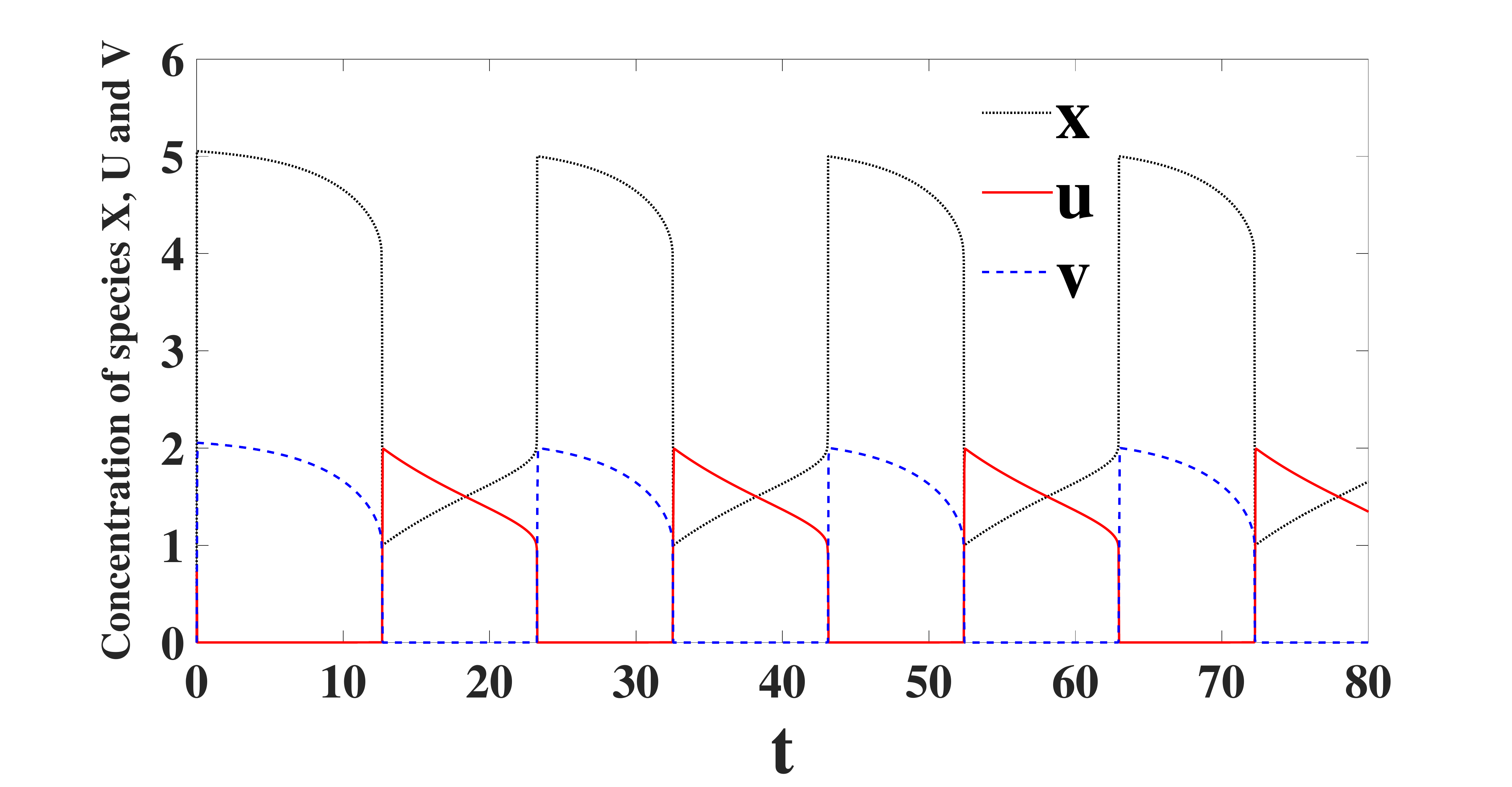}}
\subfloat[$x(0)=y(0)=4$.]{\label{fig6d}\includegraphics[width=0.5\linewidth]{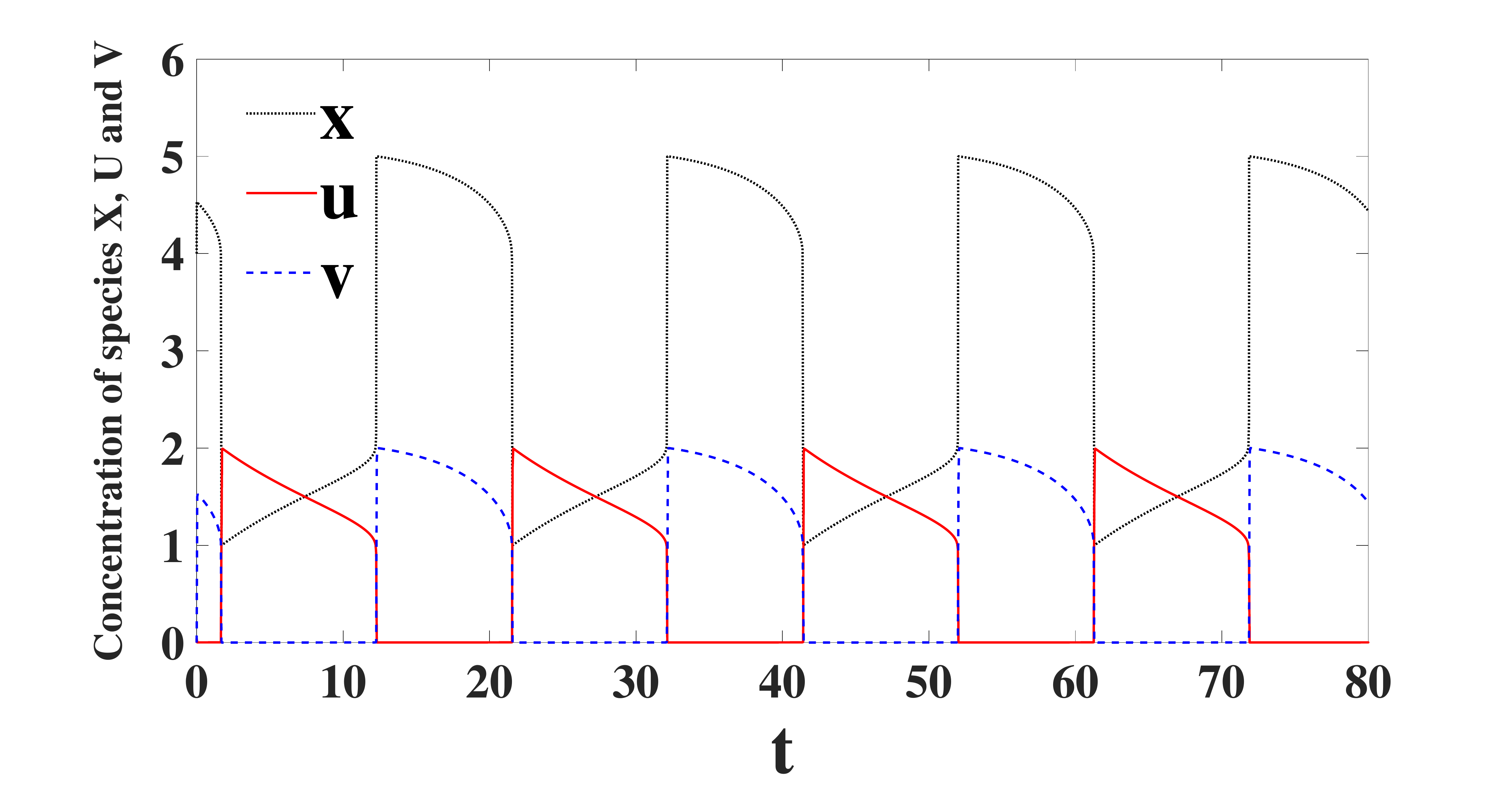}}
\caption{The oscillating behaviors of the standard chemical relaxation oscillator at a fixed initial value of $(u(0),v(0))=(0,0)$ but different initial values of $(x(0),y(0))$.}
\label{fig6}
\end{figure}

\Cref{fig5} displays the trajectory evolution starting from $4$ different initial points $A(6,6)\in\mathcal{A}_1,~ B(2,2)\in\mathcal{A}_1,~ C(0.5,0.5)\in\mathcal{A}_2$ and $D(4,4)\in\mathcal{A}_2$. They go towards the corresponding part of $\Gamma_{\epsilon_{1}}$ as \cref{pro4.4} expects. We also present the time evolution of the driving signal $X$ from these $4$ initial points in \cref{fig6}, where the time evolution of the $X$-driven symmetrical clock signals $U$ and $V$ are synchronously exhibited at a fixed initial point $(u(0),v(0))=(0,0)$. In these two figures, the other parameters $\epsilon_1,~\eta_1$ and $p$ of the standard chemical relaxation oscillator are taken the same values with those in \cref{ex3.16}. Combing \cref{fig5} and \cref{fig6} might suggest the following information.

(i) When starting from $A(6,6)$ or $B(2,2)$, the trajectory of $\Sigma_{xy}$ merges into the left part of relaxation oscillation orbit, i.e., $\Gamma_{l,\epsilon_1}$, and $x$ oscillates at the low amplitude first, leading to positive concentration of $U$ to appear, i.e., $u$ oscillating at high amplitude.  Speaking more specifically, the trajectory originating from $A(6,6)$ (may extend to the whole area $\left \{(x,y):y>5\right \}$) needs to first follow the slow manifold to reach $\Gamma_{l,\epsilon_1}$, which makes the first period of $U$ larger than the subsequent ones as \cref{fig6a} shows. However, when initial point is $B(2,2)$ (may extend to $\left \{(x,y):x<\varphi^{-1}_m(y),1<y<5\right \}$), the trajectory follows the fast manifold to reach $\Gamma_{l,\epsilon_1}$, which leads to the first period of $u$ smaller than the subsequent ones, as can be seen in \cref{fig6b}.  

(ii) \Cref{fig6c,fig6d} exhibit the corresponding cases where the trajectory merges into $\Gamma_{r,\epsilon_1}$ i.e. the right part of the relaxation oscillation orbit with positive concentration of $V$ in the first period. Furthermore, the first period of $v$ also appears larger or smaller than the subsequent ones when initial point is $C(0.5,0.5)$ or $D(4,4)$. 

\begin{remark}
For the standard chemical relaxation oscillator of \cref{eq3.15}, the position of initial point $(x(0),y(0))$ determines the oscillating position of driving signal $X$, either at high amplitude or at low amplitude, in the first period. This further drives the oscillating positions of symmetrical clock signals $U$ and $V$ in the first period, so it can control their occurrence order. Noticeably, whether $U$ (respectively, $V$) is existing or not will ``turn on" or ``turn off" the modular $\tilde{\mathcal{M}}_1$ (respectively, $\tilde{\mathcal{M}}_2$), as said in \cref{ex2.3}. Thus, the position of $(x(0),y(0))$ can finally control the computation order of $\tilde{\mathcal{M}}_1$ and $\tilde{\mathcal{M}}_2$. As an example of letting $\tilde{\mathcal{M}}_1$ execute first, the selection of initial value $(x(0),y(0))$ should ensure that the trajectory merges to $\Gamma_{l,\epsilon_1}$.  
\end{remark}


\section{Loop control and termination of molecular computations}\label{sec:ter}
In this section we apply the standard chemical relaxation oscillator of \cref{eq3.15}
 to control molecular computations periodically, and further present a termination strategy for them.

\subsection{Loop control of molecular computations}
We go back to the motivating example in \cref{sec2.3} that two reaction modules $\mathcal{M}_1$ and $\mathcal{M}_2$ are designed towards the target of making the loop iteration calculation $s_1=s_1+1$. Here, their calculations control is just taken as an application case for the standard chemical relaxation oscillator of \cref{eq3.15}. As stated in \cref{sec2.3}, these two modules need to be modified as $\tilde{\mathcal{M}}_1$ and $\tilde{\mathcal{M}}_2$, given in \cref{ex2.3}, to avoid coupling between $s_1$ and $s_2$. Then we can apply the standard chemical relaxation oscillator of \cref{eq3.15} to control their calculations through combining all related dynamical equations, i.e., combining \cref{eq3.15} with \cref{eq2.4}, which gives      
\begin{equation}\label{eq5.1}
    \begin{aligned}
         \epsilon_{1} \frac{\mathrm{d} x}{\mathrm{d} t} &=\eta_{1}(-x^3+9x^2-24x+21-y)x \ ,   &  \frac{\mathrm{d} s_{1}}{\mathrm{d} t} &= (s_{2} - s_{1})v\ , \\
        \frac{\mathrm{d} y}{\mathrm{d} t} &=\eta_{1}(x-3)y \ , &  \frac{\mathrm{d} s_{2}}{\mathrm{d} t} &= (s_{1} + s_{3} -s_{2})u\ ,  \\
   \epsilon_{1}\epsilon_{2}\frac{\mathrm{d} u}{\mathrm{d} t} &= \eta_{1}(\epsilon_{1}(p-u)-uv)\ , & \frac{\mathrm{d} s_{3}}{\mathrm{d} t} &= 0\ .  \\
\epsilon_{1}\epsilon_{2}\frac{\mathrm{d} v}{\mathrm{d} t} &= \eta_{1}(\epsilon_{1}(x-v)-uv)\ ,
    \end{aligned}
\end{equation}
The control flow goes like this: (1) the left-upper two equations produce the periodical signal $x$; (2) $x$ drives the left-lower two equations to generate symmetrical clock signals $u$ and $v$; (3) $u$ and $v$ control the iteration calculation $s_1^*=s_1^*+s_3^*$ through the right three equations according to disappearance of either $u$ or $v$.   

As an illustration, \cref{fig7a} presents the time evolution of $s_1$ and $s_2$ of the system \cref{eq5.1} starting from $(x(0),y(0),u(0),v(0),s_1(0),s_2(0),s_3(0))=  (5,5,0,0,0,0,1)$ with model parameters $\epsilon_{1}=\epsilon_{2}=0.001, \eta_{1}=0.1, p=3$, which are exactly the same with those in \cref{ex3.16}. Clearly, the curve of $s_1$ is a staircase-like function, and $s_1$ will add $1$ periodically as time goes on. This indicates that the standard chemical relaxation oscillator of \cref{eq3.15} is quite valid to periodically control the execution of computation modules $\tilde{\mathcal{M}}_1$ and $\tilde{\mathcal{M}}_2$, and finally reaches the target of performing the frequently-used loop iteration calculation $s_1=s_1+1$ encountered in many machine learning algorithms. A further look at \cref{fig7b} (actually an overlay of \cref{fig7a} and \cref{fig4}) reveals that the calculation time for these two modules are very short, i.e., time used for the curve of $s_1$ or $s_2$ to climb stair, compared to the sum of the oscillating periods of $U$ and $V$ estimated by \cref{periodEs}. The main reason is that both $\tilde{\mathcal{M}}_1$ and $\tilde{\mathcal{M}}_2$ complete calculation instruction at exponential speed, whose time consumption is far lower than $T_1+T_2\approx T_l+T_h\approx 19.6$s. There is still a large room to reduce the oscillating period of the driving signal $X$ in the standard chemical relaxation oscillator of \cref{eq3.15} for the current purpose.  

\begin{figure}[tbhp]
\centering
\subfloat[]{\label{fig7a}\includegraphics[width=0.5\linewidth]{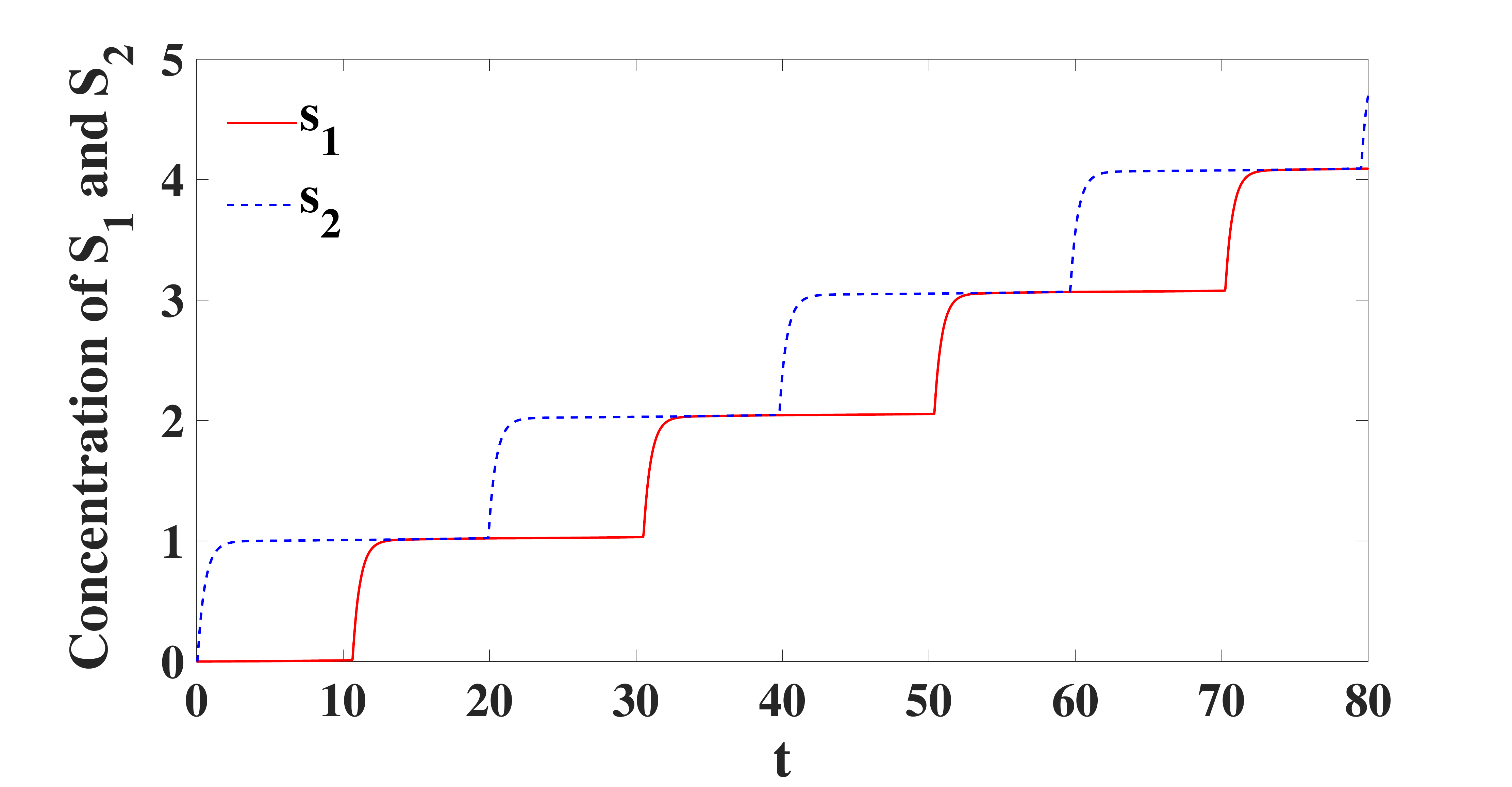}}
\subfloat[]{\label{fig7b}\includegraphics[width=0.5\linewidth]{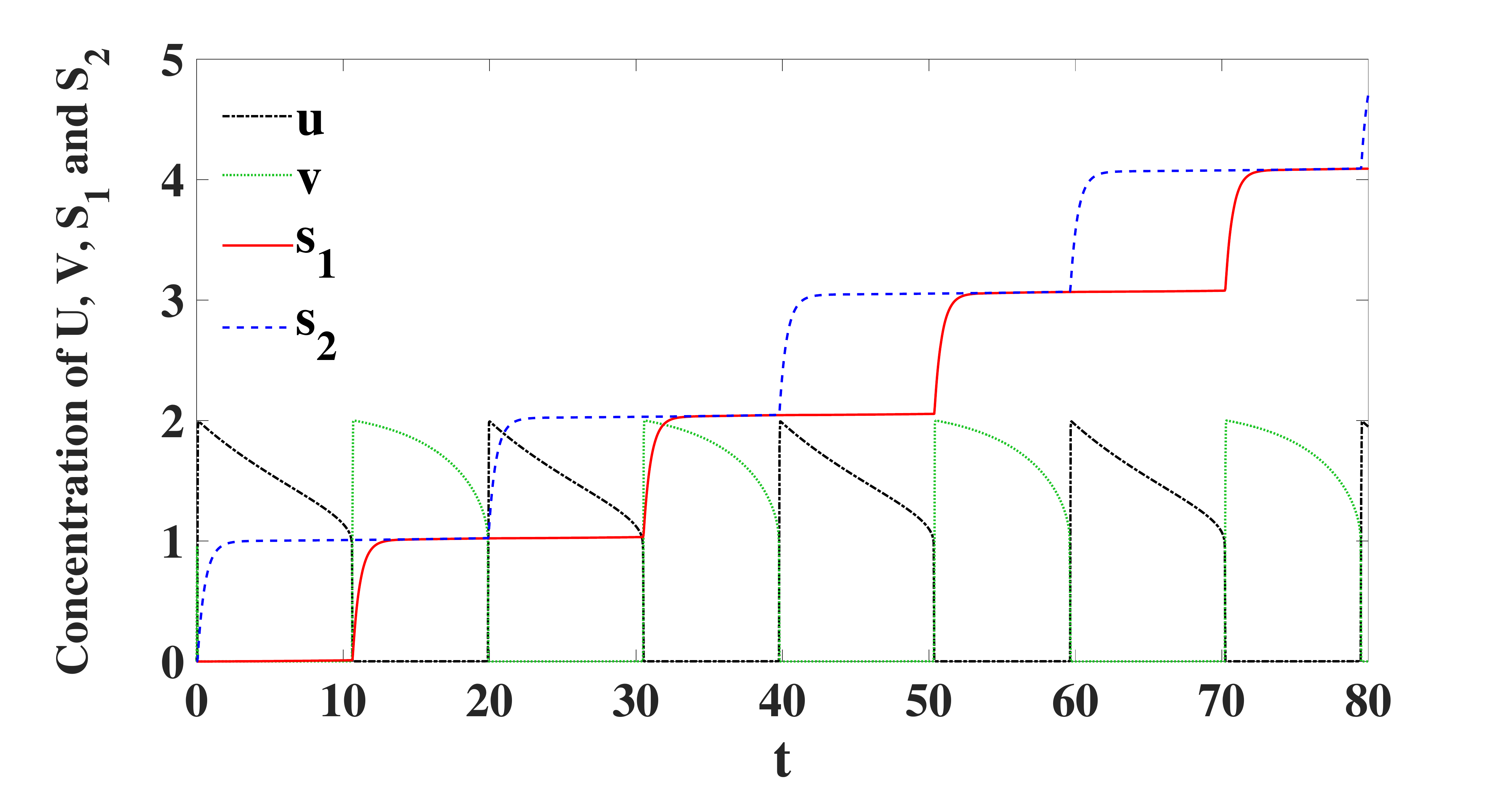}}
\caption{Time evolution of species $S_{1}$ and $S_{2}$ of the system \cref{eq5.1} in response to symmetrical clock signals $U$ and $V$: (a) without and (b) with $U$ and $V$ exhibited.}
\label{fig7}
\end{figure}


Although there are reaction networks corresponding to the dynamical equations \cref{eq3.15c}, \cref{eq3.15d} and \cref{eq2.4}, i.e., network \cref{msubtraction} with $\kappa=\eta_1/\epsilon_2$ and $\tilde{\mathcal{M}}_1$ plus $\tilde{\mathcal{M}}_2$ given in \cref{ex2.3}, respectively, the ODEs of \cref{eq5.1} are not enough to be representative of calculation made by chemical molecular since there is no reaction network corresponding to the equations of \cref{eq3.15a,eq3.15b}. It needs to find a mass-action CRN system that has the dynamics of \cref{eq3.15a,eq3.15b}, which is called CRN realization from the kinetic equations. Note that there should be many CRN realizations for the same kinetic equations, and a great deal of algorithms \cite{szederkenyi2011finding} have been developed to attain this purpose. Since how to realize a CRN from a group of kinetic equations is not the main topic of this work, we only provide a naive realization by directly transforming each monomial in \cref{eq3.15a,eq3.15b} into an abstract reaction where the species coefficients on the left and right sides are completely determined by the order of the monomial, and the parameters $\eta_1, \epsilon_1$ and $\epsilon_2$ all appear as the rate constants. The result is as follows
\begin{equation}\label{crn1}
    \begin{aligned}
4X&\xrightarrow{\eta_1/\epsilon_1}3X\ ,  & 3X&\xrightarrow{9\eta_{1}/\epsilon_{1} }4X\ , &
2X&\xrightarrow{24\eta_{1}/\epsilon_{1} }X\ , \\
X&\xrightarrow{21\eta_{1}/\epsilon_{1} }2X\ , &
X+Y&\xrightarrow{\eta_{1}/\epsilon_{1} }Y\ ,  &
X+Y&\xrightarrow{\eta_{1}}X+2Y\ , ~~~ ~~Y\xrightarrow{3\eta_{1}}\varnothing\ .
    \end{aligned}
\end{equation}

In the above reaction network, there include some very tiny parameters, such as $\epsilon_1$ and $\epsilon_2$, whose values are difficult to be evaluated precisely. As a result, it seems impossible to control the reaction rate constants as exact as assigned. Actually, these parameters, also including $\eta_1$ and even the coefficients in the S-shaped function, are not necessary to fit perfectly the assigned values. Their real values only ensure the system to generate oscillation and have a unique equilibrium which lies on the repelling part of the critical manifold and stays away from the two non-hyperbolic points.


\subsection{Loop termination of molecular computations}
We have finished the task of controlling iteration computation $s_1=s_1+1$ through the standard chemical relaxation oscillator of \cref{eq3.15}. As the driving signal $X$'s oscillation goes on, the equilibrium concentration of species $S_1$ will add $1$ periodically. This phenomenon will never stop even if there is usually a restriction of upper bound on $s_1^*$, e.g., when it is used to model the ``iteration times" in machine learning algorithm, an instruction of iteration termination should be necessary. For this reason, we introduce a new species $W$, called counter species, which on the one hand, uses its concentration $w$ to measure the difference between a given loop times $l$, an integer representing the concentration of termination species $L$, and the concentration of $S_1$; and on the other hand, can control the occurrence or termination of computation modules. 

Towards the above first purpose, we construct the following specific truncated subtraction module to generate the species $W$
\begin{align}\label{eq5.2}
  L+W &\overset{\eta_{3} }{\rightarrow} L+2W\ , &
   S_{1}+W &\overset{\eta_{3} }{\rightarrow} S_{1}\ , &  2W &\overset{\eta_{3} }{\rightarrow} W\ ,
\end{align}
whose ODEs are expressed as
\begin{subequations}\label{eq5.3}
    \begin{align}
         \frac{\mathrm{d} w}{\mathrm{d} t}&=\eta_3(l-s_{1}-w)w\ , \\
    \frac{\mathrm{d} s_{1}}{\mathrm{d} t}&=\frac{\mathrm{d} l}{\mathrm{d} t}=0\ .
    \end{align}
\end{subequations}

\begin{remark}\label{rek5.1}
The analytical solution of the ODEs \cref{eq5.3} is 
\begin{equation}
              w= \frac{(l-s_{1})w(0)}{w(0)+(l-s_{1}-w(0))e^{-\eta_3(l-s_{1})t}}\ 
\end{equation}
with $l \neq s_1$. In the case of $l> s_{1}$ and $w(0)>0$, $w$ converges exponentially to $l-s_{1}$; in the case of $l= s_{1}$, $w$ degenerates to the linear form of 
\begin{equation}\label{eq:w}
    w=\frac{w(0)}{1+\eta_3w(0)t}\ ;
\end{equation}
and in the case of $l< s_{1}$, the equilibrium is $0$. The reaction rate constant $\eta_3$ plays the role of regulating the convergence speed of $w$.   
\end{remark}

\cref{rek5.1} suggests that species $W$ will vanish as reactions \cref{eq5.2} proceed. Therefore this species can behave as a catalyst to serve for the previous second purpose, i.e., controlling the occurrence or termination of computation modules. We add it as a catalyst to each reaction of the computation modules $\tilde{\mathcal{M}}_1$ and $\tilde{\mathcal{M}}_2$, which yields two new computation modules $\hat{\mathcal{M}}_{1}$ and $\hat{\mathcal{M}}_{2}$ as
 \begin{align*}
    \hat{\mathcal{M}}_{1}:
        S_{1} + U +W &\to S_{1} + S_{2} + U +W\ , & \hat{\mathcal{M}}_{2}: S_{2} + V +W &\to S_{1} + S_{2} + V +W \ ,\\
        S_{3} + U + W&\to S_{3} + S_{2} + U + W\ , & S_{1} + V + W&\to V + W\ .\\
        S_{2} + U + W&\to U + W\ ; 
    \end{align*}
At this time the dynamics changes to be
\begin{equation}\label{eq5.6}
\begin{aligned}
   \frac{\mathrm{d} s_{1}}{\mathrm{d} t} = (s_{2} - s_{1})vw\ , 
    ~~~~\frac{\mathrm{d} s_{2}}{\mathrm{d} t} = (s_{1} + s_{3} -s_{2})uw\ ,~~~~ \frac{\mathrm{d} s_{3}}{\mathrm{d} t} = 0\ .
    \end{aligned}
\end{equation}
It is obvious that species $W$ has the same power as species $U$ and $V$, only determining the occurrence or termination of reactions, but not changing the positive equilibrium of system \cref{eq5.6}.  

 
We couple all ODEs of \cref{eq3.15,eq5.2,eq5.6}, and get the whole ODEs to be 
\begin{equation}\label{eq5.5}
\begin{aligned}
          \epsilon_{1} \frac{\mathrm{d} x}{\mathrm{d} t} &=\eta_{1}(-x^3+9x^2-24x+21-y)x \ , \ \ \ &  \frac{\mathrm{d} s_{1}}{\mathrm{d} t} &= (s_{2} - s_{1})vw\ ,  \\
         \frac{\mathrm{d} y}{\mathrm{d} t} &=\eta_{1}(x-3)y \ , 
  &  \frac{\mathrm{d} s_{2}}{\mathrm{d} t} &= (s_{1} + s_{3} -s_{2})uw\ , \\
         \epsilon_{1}\epsilon_{2}\frac{\mathrm{d} u}{\mathrm{d} t} &= \eta_{1}(\epsilon_{1}(p-u)-uv)\ ,  &   \frac{\mathrm{d} w}{\mathrm{d} t}&=\eta_{3}(l-s_{1}-w)w\ ,\\
    \epsilon_{1}\epsilon_{2}\frac{\mathrm{d} v}{\mathrm{d} t} &= \eta_{1}(\epsilon_{1}(x-v)-uv)\ , &  \frac{\mathrm{d} s_{3}}{\mathrm{d} t}& = \frac{\mathrm{d} l}{\mathrm{d} t} = 0\ ,  \\
    \end{aligned}
\end{equation}
which has the function of automatically performing loop iteration calculation and timely terminating it when calculation times is beyond a desired point. \cref{fig8} shows the loop termination results by species $W$, where \cref{fig8a} (respectively, \cref{fig8b}) simulates the case of $\eta_{3}=1$ (respectively, $\eta_3=50$), $w(0)=l(0)=4$, and other parameters and initial value information taken the same as given for \cref{fig7}. When $\eta_{3}=1$, neither $s_1$ nor $s_2$ is stable at $4$, but beyond a little. This means there is a lag to terminate $\hat{\mathcal{M}}_{1}$ and $\hat{\mathcal{M}}_{2}$. The reason is that when $s_1$ increases to be close to $l$, $w$ converges towards $0$ at a nearly linear speed as \cref{eq:w} shows, resulting in $w$ to shut down the two modules lagging behind expectation. However, when a larger $\eta_3=50$ is selected to accelerate the convergence of $w$, shown in \cref{fig7b}, the lag phenomenon is weakened and $s_1$ is stable at $4$. In practice a relatively large $\eta_3$ is thus desired to be selected. 

\begin{figure}[tbhp]
\centering
\subfloat[]{\label{fig8a}\includegraphics[width=0.5\linewidth]{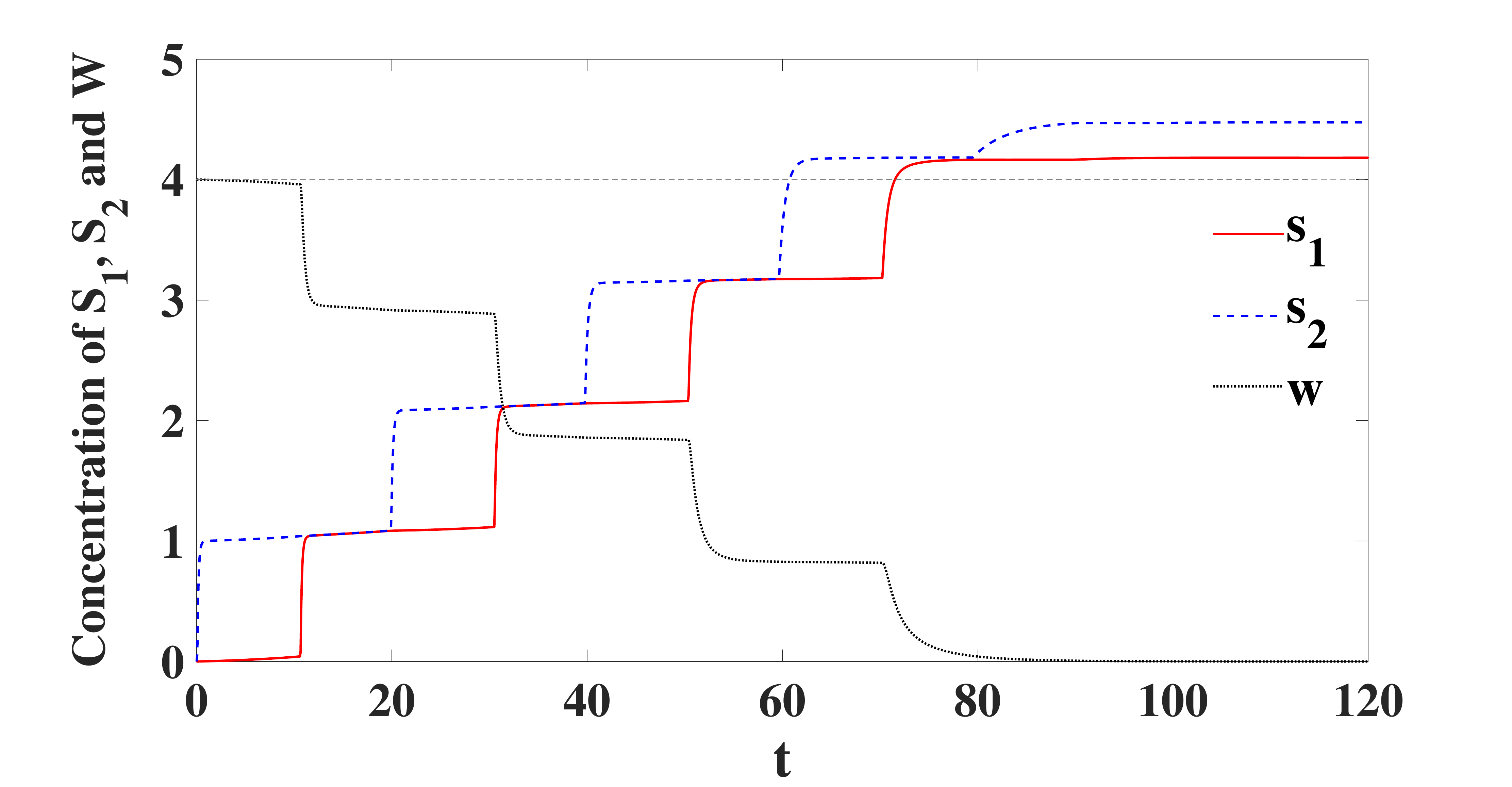}}
\subfloat[]{\label{fig8b}\includegraphics[width=0.5\linewidth]{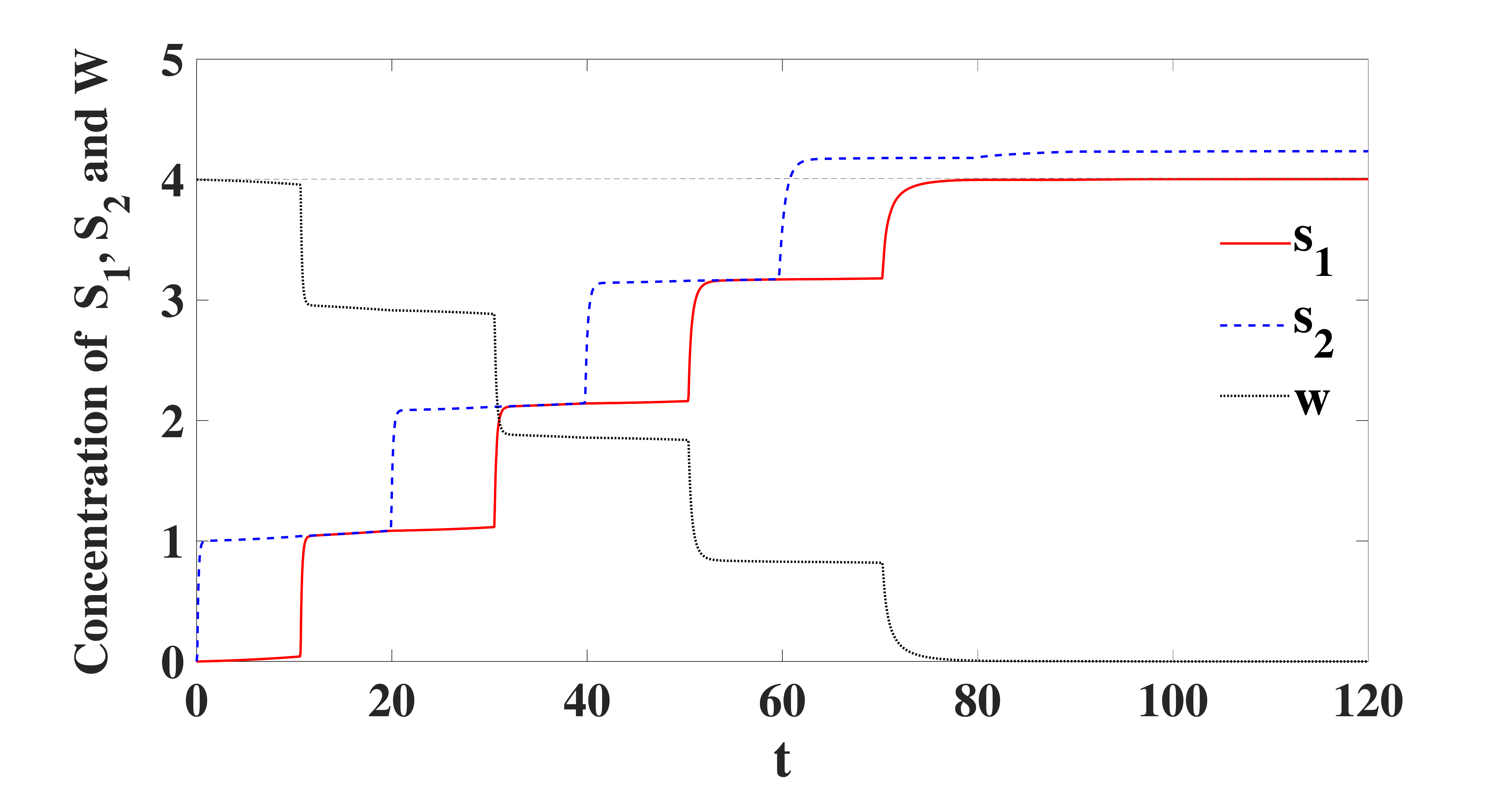}}
\caption{Loop termination of iteration computation $s_1=s_1+1$ by counter species $W$ in system \cref{eq5.5} with (a) $\eta_3=1$ and (b) $\eta_3=50$.}
\label{fig8}
\end{figure}

\begin{remark}
Our chemical relaxation oscillator design and loop termination strategy achieve the iteration computation $s_1=s_1+1$ well through chemical molecular reactions in the form of \cref{crn1} plus \cref{eq5.2} plus $\hat{\mathcal{M}}_1$ and $\hat{\mathcal{M}}_2$. This is rather a basic operation in many machine learning algorithms, and the whole modular can work as a counter CRN to control the alternation calculations of any two target modules, labeled by $\mathcal{TM}_1$ and $\mathcal{TM}_2$. As long as the symmetrical clock signals $U$ and $V$ generated by the standard chemical relaxation oscillator are fed into both $\mathcal{M}_1$, $\mathcal{M}_2$ and $\mathcal{TM}_1$, $\mathcal{TM}_2$ in parallel as catalysts, and the counter species $W$ is used to monitor the alternation times of their computations, the target calculation may be implemented. We exhibit the corresponding flowchart in \cref{fig9}. This application is quite promising, and may push the development of achieving all kinds of molecular computations.

\end{remark}




\begin{figure}[htbp]
  \centering
  \includegraphics[width=1\linewidth,scale=1.00]{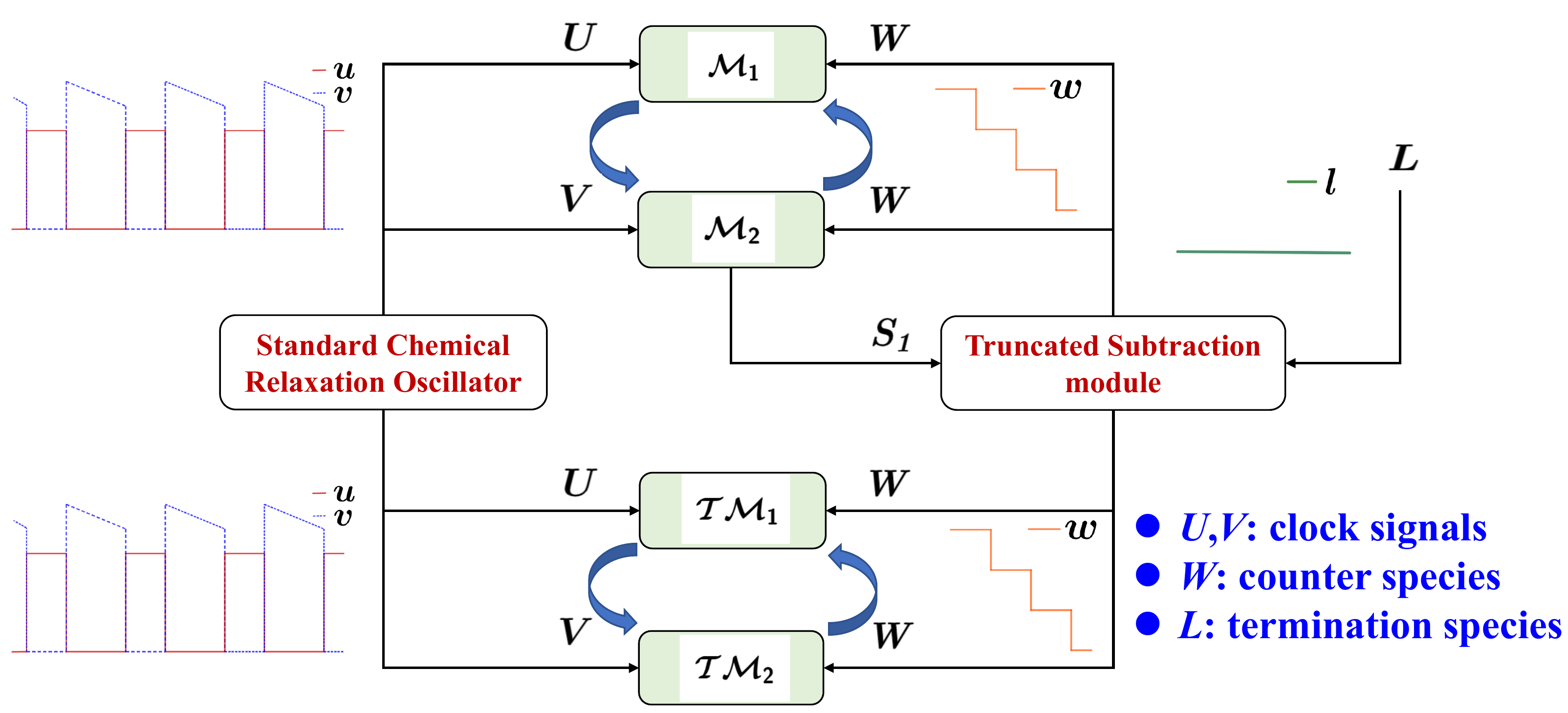}
  \caption{A schematic diagram of applying chemical oscillation-based iteration computation modules $\mathcal{M}_1$ plus $\mathcal{M}_2$ to control the alternative calculation and termination of two target modules $\mathcal{TM}_1$ and $\mathcal{TM}_2$.}
  \label{fig9}
\end{figure}

\section{Conclusions}
\label{sec:conclusions}
In this paper we develop a systematic approach to realize synchronous sequential computation with abstract chemical reactions. Our ultimate goal is to execute complex calculations in biochemical environments, and after setting the initial values of species and reaction rates, the biochemical system could run automatically to complete the target calculation task. 
For this, we design a $4$-dimensional oscillator model to generate a pair of symmetric clock signals $U$ and $V$ whose concentrations change periodically. Different from the chemical oscillators used in previous work, we construct the $4$-dimensional oscillator model based on the architecture of $2$-dimensional relaxation oscillation. We strictly analyze the dynamical properties of the oscillator model and discuss the conditions of parameters and initial values to control the period and occurrence order of $U$ and $V$. The strength of our model lies in a broad selection of initial values and a clear, easy-to-implement parameter choice.
 We demonstrate the process of module regulation under the example of $\mathcal{M}_1$ and $\mathcal{M}_2$, and give a termination strategy for the loop control. Although there is very little work that pays attention on this topic, we still believe that our consideration of loop termination makes sense for synthesizing autonomously running life. 
\par This paper actually provides guidance for implementing calculation instructions and machine learning algorithms into biochemical environments, the oscillator model we design acts as a hub connecting various parts of reaction modules corresponding to specific calculation tasks. Oscillation, especially the relaxation oscillation, plays a crucial role in this process. Different from modeling and analyzing the oscillation phenomena observed in biochemical experiments, our work takes advantage of oscillation as a means to achieve specific functions. Our $4$-dimensional oscillator model solves the task of two-module regulation well, but is a little weak when faced with tasks of three or more modules. It will be the focus of our future work to find suitable oscillation structure and design corresponding chemical oscillator model for multi-module regulation.


\bibliographystyle{siamplain}
\bibliography{references}

\end{document}